\documentclass[11pt]{article}
\usepackage{amsmath, amssymb, amscd, amsthm, amsfonts}
\usepackage{graphicx}
\usepackage{hyperref}

\newtheorem{theorem}{Theorem}[section] 
\newtheorem{corollary}[theorem]{Corollary} 
\newtheorem{lemma}[theorem]{Lemma} 
\newtheorem{condition}[theorem]{Condition} 
\newtheorem{definition}[theorem]{Definition} 
\theoremstyle{remark}

\newcommand{\R}{\mathbb{R}}
\newcommand{\C}{\mathbb{C}}

\newcommand{\p}[1]{(#1)} 

\newcommand{\pB}[1]{\Bigl(#1\Bigr)}
\newcommand{\pbb}[1]{\biggl(#1\biggr)}

\newcommand{\norm}[1]{\lVert #1 \rVert} 

\DeclareMathOperator{\re}{Re}

\newcommand{\beq}{ \begin{equation} }
\newcommand{\eeq}{ \end{equation} }

\title{Free energy fluctuation of soft spherical Sherrington--Kirkpatrick model}
\author{Ji Oon Lee\footnote{Department of Mathematical Sciences, KAIST, Daejeon, 34141, Korea \\ email: \texttt{jioon.lee@kaist.edu}} \and
Seongcheol Lee\footnote{Department of Mathematical Sciences, KAIST, Daejeon, 34141, Korea \\ email: \texttt{seongcheol97@kaist.ac.kr}}}
\date{\today}

\begin{document}

\maketitle
\begin{abstract}
We consider a soft version of the spherical Sherrington--Kirkpatrick model, where the spherical constraint is replaced by a radial confinement term. We prove that, for a real symmetric disorder matrix under suitable spectral assumptions, the free energy exhibits two distinct fluctuation regimes. In the high-temperature regime, the fluctuations are asymptotically Gaussian. In the low-temperature regime, the fluctuations are governed by the largest eigenvalue and converge to the GOE Tracy--Widom distribution. These results show that the soft radial constraint preserves the fluctuation transition of the classical spherical Sherrington--Kirkpatrick model.
\end{abstract}
\section{Introduction}

Spin glass models were introduced to describe disordered magnetic systems in which the interactions between spins are random. Among the most important examples are the Edwards--Anderson model and its mean-field analogue, the Sherrington--Kirkpatrick (SK) model; see, for example, \cite{EdwardsAnderson1975,SherringtonKirkpatrick1975,Panchenko2013}. An important variant of the SK model is the spherical Sherrington--Kirkpatrick (SSK) model in which the spin variables are constrained to lie on the sphere of radius $\sqrt N$, i.e., $S_{N-1}:=\{\sigma\in \mathbb R^N:\|\sigma\|^2=N\}$. For a symmetric disorder matrix $M$, the partition function $Z_N(\beta)$ and the free energy $F_N(\beta)$ at the inverse temperature $\beta$ are defined by
\[
Z_N(\beta) = \int_{S_{N-1}} e^{\beta \langle \sigma, M \sigma \rangle}\,d\omega_N(\sigma), \qquad F_N(\beta)=\frac1N \log Z_N(\beta),
\]
where $\omega_N$ is the normalized uniform measure on $S_{N-1}$. 

The SSK model was introduced by Kosterlitz, Thouless, and Jones in \cite{KosterlitzThoulessJones1976} and its limiting free energy was later studied rigorously through the spherical analogue of the Parisi formula in \cite{Talagrand2006}. When the disorder is given by a Wigner matrix, the fluctuations of the free energy exhibit a phase transition: in the high-temperature regime they are asymptotically Gaussian of order $N^{-1}$, whereas in the low-temperature regime they are of order $N^{-2/3}$ and converge to the Tracy--Widom distribution \cite{BaikLee2016}. More precisely, in the high-temperature regime the fluctuations are governed by all eigenvalues of the disorder matrix through a linear spectral statistic, while in the low-temperature regime they are governed by the largest eigenvalue \cite{BaikLee2016}. The corresponding random matrix inputs are the central limit theorem for linear spectral statistics and the edge universality of the largest eigenvalue; see, for example, \cite{Johansson1998,LytovaPastur2009,LeeYin2014,TracyWidom1996}.

The results on the free energy and its fluctuation, especially the sharp phase transition of the free energy, have been extended in several directions by modifying either the interaction matrix or the Hamiltonian, while keeping the spherical constraint. For the SSK model with a ferromagnetic Curie--Weiss interaction, it was shown that three fluctuation regimes arise depending on the temperature and the coupling strength \cite{BaikLee2017}. For the SSK model with deformed Wigner disorder, the high-temperature fluctuations remain Gaussian, while the low-temperature fluctuations may become Weibull depending on the edge behavior of the spectrum \cite{LeeLi2023}. For the SSK model with sparse interaction, the limiting free energy remains the same as in the Wigner case, while the fluctuations change with the sparsity level \cite{KimLee2025}; the relevant local laws and edge universality results for sparse random matrices were developed, for example, in \cite{LeeSchnelli2018}. For a recent survey of free-energy fluctuations in the SK, SSK, and related models, we refer to \cite{CollinsWoodfinLe2025}.

In this paper, we consider a different type of modification, known as the soft spherical Sherrington--Kirkpatrick (SSSK) model. While the norm of the spin vector is fixed in the SSK model, we introduce a radial constraint for the spin vector, given by a convex confining function. More precisely, the partition function and the free energy of the SSSK model are defined as follows:
\[
Z_N(\beta) = \int_{\mathbb R^N} \exp\left\{ \beta\langle v,Mv\rangle - \beta N f\left(\frac{\|v\|^2}{N}\right) \right\}\,\prod_{i=1}^N dv_i, \qquad F_N(\beta)=\frac1N\log Z_N(\beta).
\]
Here, $M$ is a real symmetric matrix and $\prod dv_i$ denotes the Lebesgue measure on $\R^N$. (For the precise assumptions on the confinement function $f$, see Definition \ref{def:partition-function}.) In this model, the norm of the spin vector is no longer fixed, and thus the model considered here may be viewed as a soft version of the SSK model. Indeed, if we replace $f$ by $\kappa\phi$, where $\phi$ is strongly convex and uniquely minimized at $1$, then, at least formally, the limit $\kappa\to\infty$ forces $\|v\|^2/N$ to concentrate near $1$ and recovers the usual spherical constraint at the level of the radial variable.

The soft, spherical model of spin glass arises in the study of Langevin dynamics for spherical spin glasses. In the work of Ben Arous, Dembo, and Guionnet \cite{BenArousDemboGuionnet2001}, the spherical constraint is implemented through a radial confining potential, and the resulting Gibbs measure appears as the invariant measure of the dynamics. They studied the large-$N$ dynamics and its long-time behavior, including the aging phenomenon, and also analyzed the corresponding equilibrium measure and its limiting free energy. In particular, when $M$ is a GOE matrix and $f(x)=\frac{\kappa}{2}x^2$ for a fixed $\kappa>0$, 
\begin{align} \label{eq:GOE-Flimit}
F_N(\beta) \to F(\beta)
=
\begin{cases}
\displaystyle
\frac14+\frac12\log\left(\frac{\pi}{\beta}\right)
+\frac14\log\left(\frac{2\beta}{\kappa-2\beta}\right),
& 0<\beta\leq\frac{\kappa}{4}, \\[1.2em]
\displaystyle
-\frac14
+\frac12\log\left(\frac{\pi}{\beta}\right)
+\frac{2\beta}{\kappa},
& \beta>\frac{\kappa}{4}.
\end{cases}
\end{align}
Related dynamical aspects of spherical spin glasses have also been extensively studied, including aging phenomena and universality beyond Gaussian disorder \cite{CugliandoloDean1995,DemboGheissari2021}. More broadly, ideas and dynamical mean-field methods developed in the study of disordered systems and spin glasses have found applications in optimization, ecology, and learning theory; see \cite{Cugliandolo2024} for a recent review.

Similar soft-spin models also arise in equilibrium statistical mechanics. For example, Barra et al. \cite{BarraGenoveseGuerraTantari2014} studied a solvable mean-field spin glass with unbounded Gaussian spins stabilized by a nonlinear confinement, providing another soft-spin analogue of constrained spin-glass models.

The aim of the present paper is to prove the free energy fluctuation results for the soft spherical models; while the limiting free energy of the GOE model with quadratic confinement was proved in \cite{BenArousDemboGuionnet2001}, to our best knowledge, the fluctuations of the free energy have never been proved. In this paper, we show that the Gaussian-to-Tracy--Widom fluctuation transition of the classical spherical Sherrington--Kirkpatrick model, proved in \cite{BaikLee2016}, persists under the soft radial constraint. In the simplest case where $M$ is a GOE matrix and $f(x)=\kappa x^2 /2$ for a fixed $\kappa>0$, we obtain the following result:

\begin{theorem}[GOE case with quadratic confinement] \label{thm:specific-theorem}
Let $M$ be an $N\times N$ GOE matrix normalized so that $M_{ij}\sim \mathcal N(0,N^{-1})$ for $i<j$ and $M_{ii}\sim \mathcal N(0,2N^{-1})$. Let $f(x)=\kappa x^2 /2$ for a fixed $\kappa>0$. Then the following statements hold as $N\to\infty$. Here $F(\beta)$ is defined in \eqref{eq:GOE-Flimit}, and all convergences are in distribution.  
\begin{enumerate}
    \item[\textup{(i)}] If $0<\beta<\kappa/4$, then
\[
N\bigl(F_N(\beta)-F(\beta)\bigr) \Rightarrow \mathcal N(\ell_\kappa(\beta),\sigma_\kappa^2(\beta)),
\]
where $\mathcal N(m,s^2)$ denotes the Gaussian law with mean $m$ and variance $s^2$. Here,
\[
\ell_\kappa(\beta) = \frac14\log\left(\frac{1-4\beta/\kappa}{1-2\beta/\kappa}\right) -\frac12\log2, \qquad \sigma_\kappa^2(\beta) = -\frac12\log\left(\frac{1-4\beta/\kappa}{1-2\beta/\kappa}\right).
\]
    \item[\textup{(ii)}] If $\beta>\kappa/4$, then
\[
\frac{N^{2/3}}{\,2\beta/\kappa-\frac12\,} \bigl(F_N(\beta)-F(\beta)\bigr) \Rightarrow TW_1,
\]
where $TW_1$ is the GOE Tracy--Widom distribution.
\end{enumerate}
\end{theorem}

Theorem \ref{thm:specific-theorem} is a special case of our main result, Theorem \ref{thm:general-theorem}, where the free energy fluctuation results are generalized in two different directions; the results in Theorem \ref{thm:specific-theorem} hold with suitable changes, even when the disorder $M$ is not a GOE matrix and the confining potential $f$ is not quadratic.

Our proof of the main result is based on the integral representation of the partition function (Corollary \ref{cor:soft-integral-representation}) and the steepest-descent analysis. The main technical difficulty in the analysis stems from the fact that the new `radial variable' ($t$ in Corollary \ref{cor:soft-integral-representation}) is introduced in the integral representation formula due to the soft confining potential, and it is coupled with the inverse temperature $\beta$. In the analysis of the SSK model, where the radial variable is absent, the free energy is approximated by a linear statistic of the eigenvalues in the high-temperature regime, while in the low-temperature regime the leading contribution comes from the top eigenvalue. In contrast, due to the coupling between the radial variable and the inverse temperature, we find that for any $\beta$, the partition function is decomposed into a `high temperature' integral where all eigenvalues contribute and a `low temperature' integral where the contribution from the largest eigenvalue dominates. For more detail, see Section~\ref{sec:integral}.

The study of the free energy fluctuation can be applied to the signal detection problem for spiked Wigner matrices. Suppose that matrix-type data is given, which is of the form
\[
Y_{ij} = \sqrt{N} M_{ij} + \sqrt{\frac{\lambda}{N}} v_i v_j.
\]
The matrix $Y$ is a spiked Wigner matrix, where $v$ is the signal (or spike), $M$ the noise, and $\lambda$ the signal-to-noise ratio (SNR). In the detection problem, the key question is whether the data matrix is drawn from a spiked distribution ($\lambda = \lambda_0$ for a certain $\lambda_0 > 0$) or an unspiked (null) distribution ($\lambda = 0$). It is well-known that the optimal statistic to test the presence of the signal is the Radon--Nikodym derivative, or the likelihood ratio.
Fix $\alpha>0$. Let $M$ be a GOE matrix normalized as in Theorem~\ref{thm:specific-theorem}, and let $v$ be independent of $M$ with density $C_N(\alpha)\exp(-\alpha\|v\|^4/N)$. Then, the likelihood ratio $\mathcal{L}(\lambda_0)$ under the null is given by
\[
\mathcal{L}(\lambda_0) := C_N(\alpha) \int_{\mathbb R^N}
    \exp\left\{\frac{\sqrt{\lambda_0}}{2} \sum_{i,j=1}^N M_{ij}v_iv_j
    - \frac{(\lambda_0 +4\alpha) \| v \|^4}{4N}  \right\}
    \prod_{i=1}^N dv_i.
\]
(Note that the normalization constant $C_N(\alpha) = 2 \pi^{-N/2} \alpha^{N/4} N^{-N/4} \Gamma(N/4)^{-1} \Gamma(N/2)$.) For fixed $0<\lambda_0<4\alpha$, applying Theorem~\ref{thm:specific-theorem} and Stirling's formula gives, under the null,
\[
\log\mathcal{L}(\lambda_0) \Rightarrow \mathcal{N} \left( \frac{1}{4} \log (1-\frac{\lambda_0}{4\alpha}), -\frac{1}{2} \log (1-\frac{\lambda_0}{4\alpha}) \right).
\]
Applying a standard argument (see, e.g., Corollary 5 of \cite{AlaouiJordan2018}), it is also possible to find that the sum of the Type-I error and the Type-II error converges to $\mathrm{erfc} (\frac{1}{4} \sqrt{- \log (1-\frac{\lambda_0}{4\alpha})} )$. More general cases can also be analyzed by applying Theorem \ref{thm:general-theorem}, provided that $M_{ij}$ are Gaussian.

\subsection{Main result}
In this subsection, we state our main result, Theorem \ref{thm:general-theorem}. We begin by precisely defining the soft spherical Sherrington--Kirkpatrick model studied in this paper.

\begin{definition}\label{def:partition-function}
Let $M$ be an $N\times N$ real symmetric matrix and $v = \p{v_1,...,v_N} \in \R^N$. Let $f:\mathbb R_+\to\mathbb R$ be a strongly convex $C^1$ function such that $f(x)/x\to\infty$ as $x\to\infty$.
{
Whenever a complex saddle-point expansion is used, we assume the corresponding local holomorphic extension specified in Theorem~\ref{thm:general-theorem}.
}
We define the partition function $Z_N$ and the associated free energy $F_N$ by
\begin{equation}
    Z_N \equiv Z_N(\beta)=\int_{\mathbb R^N}
    \exp\left\{\beta\sum_{i,j=1}^N M_{ij}v_iv_j
    -\beta N f\left(\frac1N\sum_{i=1}^N v_i^2\right)\right\}
    \prod_{i=1}^N dv_i,
    \qquad
    F_N(\beta)=\frac1N\log Z_N,
\end{equation}
where $\prod_{i=1}^N dv_i$ denotes the Lebesgue measure on $\mathbb R^N$. In the rest of the paper we write $f$ also for its local holomorphic extension.
\end{definition}

Let $\lambda_1\geq\lambda_2\geq\cdots\geq\lambda_N$ denote the eigenvalues of $M$. Due to the rotational symmetry, the free energy depends only on the spectrum of $M$. Following \cite{BaikLee2016}, we prove our main result for any disorder matrix $M$ satisfying certain conditions on its spectrum. These conditions are about the limiting spectral measure, the rigidity of eigenvalues, the asymptotic normality of the linear statistics, and the Tracy--Widom limit of the largest eigenvalue. As shown in \cite[Section 3]{BaikLee2016}, these conditions are satisfied by a large class of random matrices, including Wigner matrices, orthogonal invariant ensembles, and real sample covariance matrices. For general background on these topics, we refer to \cite{AndersonGuionnetZeitouni2010,PasturShcherbina2011}.

We now list the conditions we will assume in this paper. The first is the regular limiting behavior of the empirical spectral measure.
\begin{condition}[Regularity of measure] \label{cond:regularity-measure}
The empirical spectral measure $\nu_N :=\frac{1}{N} \sum^N_{j=1} \delta_{\lambda_j}$ of $M$ converges weakly in probability to a probability measure $\nu$ satisfying the following properties:
    \begin{itemize}
    \item[-] $\nu$ is supported on an interval $[C_{-},C_{+}]$, and its density is positive on $(C_{-},C_{+})$.
    \item[-] $\nu$ is absolutely continuous and $\frac{d\nu}{dx}$ exhibits square-root decay at the upper edge, i.e.,
    \begin{equation}\label{eq:s-nu}
        \frac{d\nu}{dx}(x) = s_\nu\sqrt{C_{+}-x}(1+ O(C_{+}-x)) \;\;\; \mbox{as} \; x\nearrow C_{+},
    \end{equation}
    for some $s_\nu > 0$.
    \end{itemize}
\end{condition}

The limiting measure describes the global distribution of the eigenvalues. For the analysis of free-energy fluctuations, however, we also need a much finer control on the location of individual eigenvalues. This is provided by the rigidity estimate below.

\begin{condition}[Rigidity of eigenvalues] \label{cond:rigidity-eigenvalues}
For a positive integer $k \in [1,N]$, let $\widehat{k} := \min\{k, N+1-k\}$. Let $\xi_k$ be the classical location defined by
    \begin{equation} \label{eq:classical-location}
        \int^\infty_{\xi_k} d\nu = \frac{1}{N}\left(k -\frac{1}{2}\right).
    \end{equation}
Uniformly for $1\le k\le N$,
    \begin{equation} \label{eq:C-gamma-bound}
        |\lambda_k - \xi_k| \prec \widehat{k}^{-1/3}N^{-2/3}.
    \end{equation}
Moreover, there exists a constant $C>1$, independent of $N$, such that
    \begin{equation}
        C^{-1}k^{2/3}N^{-2/3}
        \le C_+-\xi_k
        \le Ck^{2/3}N^{-2/3}.
    \end{equation}
\end{condition}

The rigidity condition ensures that each eigenvalue stays close to its classical location up to a very small error. For generalized Wigner matrices, such eigenvalue rigidity follows from the local semicircle law; see \cite{ErdosYauYin2012}. This strong form of control allows us to approximate sums over eigenvalues by integrals with respect to the limiting measure $\nu$, with a quantitative rate of convergence. In particular, it enables us to obtain precise estimates for linear spectral statistics, such as the resolvent $\sum (z-\lambda_i)^{-1}$ and the logarithmic sum $\sum \log (z-\lambda_i)$, which play a central role in our analysis. The following corollaries make this approximation explicit.

\begin{corollary}[Bulk approximation]
\label{cor:bulk-approximation}
Fix $\delta>0$. Uniformly for $z\ge C_++\delta$,
    \begin{align}
        \left| \frac{1}{N}\sum_{i=1}^N \frac{1}{z-\lambda_i} - \int_{C_{-}}^{C_{+}} \frac{d\nu(x)}{z-x} \right| &= O_\prec(N^{-1}), \\
        \left|\frac1N\sum_{i=1}^N\log(z-\lambda_i) - \int_{C_-}^{C_+}\log(z-x)\,d\nu(x)\right|&=O_\prec(N^{-1}).
    \end{align}
\end{corollary}

\begin{corollary}[Edge approximation]
\label{cor:edge-approximation}
Let $0<\epsilon<1/3$. Uniformly for $z$ satisfying $ N^{-1 + \epsilon} \leq z - \lambda_1 \prec N^{-2/3}$, we have
    \begin{align}
        \left| \frac{1}{N}\sum_{i=1}^N \frac{1}{z-\lambda_i} - \int_{C_{-}}^{C_{+}} \frac{d\nu(x)}{C_+-x} \right| = o(1).
    \end{align}
Also, let $m$ be any fixed positive integer. Uniformly for $z$ satisfying $N^{-m} \leq z - \lambda_1 \prec N^{-2/3}$, we have
    \begin{align}
        \frac1N\sum_{i=1}^N\log(z-\lambda_i)
        &=\int_{C_-}^{C_+}\log(C_+-x)\,d\nu(x)
        +(z-C_+)\int_{C_-}^{C_+}\frac{d\nu(x)}{C_+-x}
        +O_\prec(N^{-1}).
    \end{align}
\end{corollary}

The third condition is about the linear statistics of the eigenvalues.
\begin{condition}[Linear statistics of the eigenvalues] \label{cond:linear-statistics}
For every function $\varphi: \mathbb{R} \to \mathbb{R}$ that is analytic in an open neighborhood of $[C_-,C_+]$ and has compact support, the random variable
    \begin{equation} \label{eq:linear-statistics}
        \mathcal{N}_\varphi := \sum_i \varphi(\lambda_i) - N\int_{C_{-}}^{C_{+}}\varphi(\lambda)d\nu(\lambda)
    \end{equation}
converges in distribution to a Gaussian random variable. The mean and the variance of this Gaussian random variable are denoted by $\rm{Mean}(\varphi)$ and $\rm{Var}(\varphi)$, respectively, and they depend only on $\varphi$ restricted to the support of $\nu$.
\end{condition}
For Wigner matrices, central limit theorems of this type were established in \cite{LytovaPastur2009}; see also \cite{Johansson1998} for unitary/Hermitian invariant ensembles and \cite{BaiSilverstein2004} for sample covariance matrices. See also Section 3 of \cite{BaikLee2016} for explicit formulas of $\rm{Mean}(\varphi)$ and $\rm{Var}(\varphi)$.

The fourth condition concerns the convergence to the Tracy--Widom distribution of the largest eigenvalue.
\begin{condition}[Tracy--Widom limit of the largest eigenvalue] \label{cond:tracy-widom}
Let $s_\nu$ be the constant appearing in \eqref{eq:s-nu}. The rescaled largest eigenvalue $(s_\nu\pi)^{2/3}N^{2/3}(\lambda_1-C_+)$ converges in distribution to the GOE Tracy--Widom distribution.
\end{condition}
For Wigner matrices, Condition~\ref{cond:tracy-widom} follows under standard tail assumptions from edge universality results such as \cite{LeeYin2014}; the limiting GOE law was introduced in \cite{TracyWidom1996}.

We are now ready to state our main result.

\begin{theorem}[General spectral case]\label{thm:general-theorem}
Suppose that $M$ is an $N\times N$ real symmetric random matrix satisfying Conditions~\ref{cond:regularity-measure}, \ref{cond:rigidity-eigenvalues}, \ref{cond:linear-statistics}, and \ref{cond:tracy-widom}, and that $f:\mathbb R_+\to\mathbb R$ is a strongly convex $C^1$ function such that $\frac{f(x)}{x}\to\infty \text{ as }x\to\infty$. Let $C_+$ be the right edge of the limiting spectral measure $\nu$, and set $\beta_c=\frac12\int_{C_-}^{C_+}\frac{d\nu(x)}{C_+-x}$. For $z\in \R$, let $\tau_z := \operatorname*{arg\,max}_{t \ge 0} \bigl( tz - f(t) \bigr)$ and define $V(z):=z\tau_z-f(\tau_z)$. Denote $\beta^* := \frac{\beta_c}{\tau_{C_+}}$ with the convention that $\beta^* = \infty$ when $\tau_{C_+} = 0$. For each $\beta$, assume moreover that $f$ admits a holomorphic extension to a complex neighborhood of the corresponding radial saddle point: a neighborhood of $\tau_{\widehat\gamma(\beta)}$ when $0<\beta<\beta^*$, and a neighborhood of $\tau_{C_+}$ when $\beta\geq\beta^*$. Then the following hold as $N\to\infty$.

\begin{enumerate}
\item[(i)]
Define
\begin{equation}
    F(\beta)
    =
    \begin{cases}
    \beta V(\widehat\gamma)
    +\frac12\log\left(\frac{\pi}{\beta}\right)
    -\frac12\int_{C_-}^{C_+}\log(\widehat\gamma-x)\,d\nu(x),
    & 0<\beta\leq\beta^*,\\[1ex]
    \beta V(C_+)
    +\frac12\log\left(\frac{\pi}{\beta}\right)
    -\frac12\int_{C_-}^{C_+}\log(C_+-x)\,d\nu(x),
    & \beta>\beta^*,
    \end{cases}
\end{equation}
where, in the first case, $\widehat\gamma \geq C_+$ is the unique solution of
\begin{equation}\label{eq:gammahat-definition}
    \beta\tau_{\widehat\gamma}
    =
    \frac12\int_{C_-}^{C_+}\frac{d\nu(x)}{\widehat\gamma-x}.
\end{equation}
Then
\begin{equation}
    F_N(\beta) \to F(\beta)
\end{equation}
in probability.
\item[(ii)] If $0<\beta<\beta^*$, then
\begin{equation}
    N\bigl(F_N(\beta)-F(\beta)\bigr)
    \Rightarrow
    \mathcal N\bigl(\ell(\beta),\sigma^2(\beta)\bigr),
\end{equation}
where $F(\beta)$ is defined in (i) and
\begin{equation}
    \ell(\beta)
    =
    -\frac12\rm{Mean}(\varphi)
    -
    \frac12\log\left( 1 - \frac{f''(\tau_{\widehat\gamma})}{2\beta} \int^{C_+}_{C_-}\varphi''(x)\,d\nu(x)  \right),
    \qquad
    \sigma^2(\beta)
    =
    \frac14\rm{Var}(\varphi).
\end{equation}
Here, $\operatorname{Mean}$ and $\operatorname{Var}$ denote the mean and variance functionals defined in Condition~\ref{cond:linear-statistics} and $\varphi$ is any compactly supported function that is analytic in a neighborhood of $[C_-,C_+]$ and agrees with $x\mapsto\log(\widehat\gamma-x)$ on that neighborhood.
\item[(iii)]
If $\beta>\beta^*$, then
\begin{equation}
    \frac{(s_\nu \pi N)^{2/3}}{\beta\tau_{C_+}-\beta_c}
    \bigl(F_N(\beta)-F(\beta)\bigr)
    \Rightarrow TW_1.
\end{equation}
Here, $s_\nu$ is defined in Condition~\ref{cond:regularity-measure} and $F(\beta)$ is defined in (i).
\end{enumerate}
\end{theorem}

Lastly, we introduce notation that will be used throughout the paper.

\begin{definition}[Stochastic domination] \label{def:stochastic-domination}
Let
    \begin{align}
        X = (X^{(N)}(u) : N \in \mathbb{N}, u \in U^{(N)}), Y = (Y^{(N)}(u) : N \in \mathbb{N}, u \in U^{(N)})
    \end{align}
be two families of random variables, where $Y^{(N)}(u)$ are nonnegative and $U^{(N)}$ is a possibly $N$-dependent parameter set. We say that $X$ is stochastically dominated by $Y$, uniformly in $u$, if for all small $\epsilon > 0$ and large $D > 0$, there exists $N_0(\epsilon,D)$ such that
    \begin{align}
        \sup_{u \in U^{(N)}} \mathbb{P}[|X^{(N)}(u)| > N^\epsilon Y^{(N)}(u)] \leq N^{-D}
    \end{align}
for all $N \geq N_0(\epsilon, D)$. If $X$ is stochastically dominated by $Y$, uniformly in $u$, we use the notation $X \prec Y$ or $X = O_\prec (Y)$.
\end{definition}

\begin{definition}[Overwhelming probability]
We say that an $N$-dependent event $\Omega_N$ holds with overwhelming probability if, for any given $D > 0$, there exists $N_0 > 0$ such that
    \begin{align}
        \mathbb{P}(\Omega_N^c) \leq N^{-D}
    \end{align}
for all $N \geq N_0$.
\end{definition}

\subsection{Organization of the paper}

The rest of the paper is organized as follows. In Section~\ref{sec:integral}, we derive an integral representation of the partition function and decompose it into two parts, $Z_N^H(\beta)$ and $Z_N^L(\beta)$, corresponding to the bulk and edge contributions, respectively. In Section~\ref{sec:high}, we analyze the high-temperature regime and prove that the leading contribution comes from $Z_N^H(\beta)$. In Section~\ref{sec:low}, we study the low-temperature regime, where $Z_N^L(\beta)$ gives the leading contribution and the largest eigenvalue governs the fluctuation. In Section~\ref{sec:proof}, we combine these estimates to prove the main results, discuss the regularity of the limiting free energy, and explain the relation with the hard spherical constraint. Proofs of technical estimates are collected in Appendix.

\section{Integral representation} \label{sec:integral}

Our proof follows the steepest descent method (saddle-point method) used by Baik and Lee for the spherical Sherrington--Kirkpatrick model {\cite{BaikLee2016}}. The starting point of their analysis is a contour integral representation of the spherical partition function. We first recall this representation, Lemma 1.3 of \cite{BaikLee2016}, since it will be used as the basic input in our analysis. Here and throughout this paper, the logarithm is taken in the principal branch.

\begin{lemma}\label{lem:integral-representation}
Let $M$ be an $N\times N$ real symmetric matrix with eigenvalues $\lambda_1 \geq \cdots \geq \lambda_N$ and let $d\omega_N$ be the normalized uniform measure on $S_{N-1}:=\{\sigma\in\mathbb R^N:\|\sigma\|^2=N\}$. Then
    \begin{align}
        \int_{S_{N-1}}e^{\beta\langle\sigma,M\sigma\rangle}d\omega_N(\sigma) = C_N\int_{\gamma-\mathrm{i}\infty}^{\gamma+\mathrm{i}\infty}e^{\frac{N}{2}G_0(z)}dz, \nonumber \\
        G_0(z) = 2\beta z - \frac{1}{N}\sum_i \log(z-\lambda_i),
    \end{align}
where $\gamma$ is any constant satisfying $\gamma>\lambda_1$ and the integration contour is the vertical line from $\gamma-\mathrm{i}\infty$ to $\gamma+\mathrm{i}\infty$, and
    \begin{equation}
        C_N = \frac{\Gamma(N/2)}{2\pi \mathrm{i} (N\beta)^{N/2-1}}.
    \end{equation}
Here $\Gamma(z)$ denotes the Gamma function.
\end{lemma}

The proof of Lemma \ref{lem:integral-representation} can be found in Section 4 of \cite{BaikLee2016}; for completeness, we also provide it in Appendix~\ref{app:integral-representation}. This gives the following representation of the partition function.

\begin{corollary}\label{cor:soft-integral-representation}
The partition function $Z_N$ of the soft SSK model admits the following representation. For any $\gamma>\lambda_1$,
\begin{align} \label{eq:partition-function}
     Z_N
    &=
    \int_{\mathbb R^N}
    e^{\beta\left(\sum_{i,j=1}^N M_{ij}v_i v_j
    -
    Nf\left(\frac1N\sum_{i=1}^N v_i^2\right)\right)}
    \prod_{i=1}^N dv_i \nonumber  \\
    &=
    -\frac{\mathrm{i} N}{2}
    \left(\frac{\pi}{\beta}\right)^{\frac N2-1}
    \int_0^\infty
    e^{-\beta N f(t)}
    \int_{\gamma-\mathrm{i}\infty}^{\gamma+\mathrm{i}\infty}
    e^{\frac N2 G(z;t)}\,dz\,dt,
\end{align}
where
\[
G(z;t) = 2\beta t z - \frac1N\sum_{i=1}^N \log(z-\lambda_i).
\]
\end{corollary}

\begin{proof}
We use the polar decomposition
\[
v=\sqrt{t}\,\sigma,\qquad t=\frac{\|v\|^2}{N},\qquad \sigma\in S_{N-1}.
\]
Then
\[
\langle v,Mv\rangle=t\langle \sigma,M\sigma\rangle \qquad \text{and} \qquad dv=C_N^{\mathrm{Leb}}t^{\frac N2-1}\,dt\,d\omega_N(\sigma),
\]
where $C_N^{\mathrm{Leb}}=\frac{\pi^{N/2}N^{N/2}}{\Gamma(N/2)}$. Hence
\[
Z_N(\beta) = C_N^{\mathrm{Leb}} \int_0^\infty t^{\frac N2-1}e^{-\beta Nf(t)} \left( \int_{S_{N-1}} e^{\beta t\langle \sigma,M\sigma\rangle} d\omega_N(\sigma) \right)dt.
\]
Applying Lemma~\ref{lem:integral-representation} to the inner spherical integral with inverse temperature $\beta t$ gives the result.
\end{proof}

For each fixed value of the radial variable $t$, the contour integral in \eqref{eq:partition-function} has the same form as the contour representation for the spherical model, with inverse temperature $\beta$ replaced by the effective inverse temperature $\beta t$. In the steepest descent analysis of Baik--Lee~\cite{BaikLee2016}, the comparison between the inverse temperature $\beta$ and the critical inverse temperature $\beta_c$ defined by
\begin{equation}\label{eq:beta-c}
    \beta_c := \frac12 \int_{C_-}^{C_+} \frac{d\nu(k)}{C_+ - k}
\end{equation}
determines whether the main contribution comes from the spectral bulk or from the spectral edge. Note that $\beta_c>0$ is well defined from Condition \ref{cond:regularity-measure}. In the region $0<\beta<\beta_c$, the saddle point stays away from the upper edge $C_+$ implying that the contribution comes from the spectral bulk. In contrast, in the region $\beta>\beta_c$ the dominant contribution is expected to come from the vicinity of the largest eigenvalue.

In the present model, since $t$ ranges over $(0,\infty)$, the effective inverse temperature $\beta t$ ranges from $0$ to $\infty$. Hence both regions may appear in the radial integral, and it is natural to split the integral at the value of $t$ for which $\beta t$ reaches the critical inverse temperature $\beta_c$, i.e., $t=\frac{\beta_c}{\beta}$. We therefore split the partition function into two parts,
\[
Z_N(\beta)=Z_N^H (\beta)+Z_N^L(\beta),
\]
where
\begin{equation} \label{eq:H-N}
    Z_N^H(\beta)
    =
    -\frac{\mathrm{i}N}{2}
    \left(\frac{\pi}{\beta}\right)^{N/2-1}
    \int_0^{\beta_c/\beta}
    e^{-\beta N f(t)}
    \left(
        \int_{\gamma-\mathrm{i}\infty}^{\gamma+\mathrm{i}\infty}
        e^{\frac{N}{2}G(z;t)}\,dz
    \right)dt,
\end{equation}
and
\begin{equation} \label{eq:L-N}
    Z_N^L(\beta)
    =
    -\frac{\mathrm{i}N}{2}
    \left(\frac{\pi}{\beta}\right)^{N/2-1}
    \int_{\beta_c/\beta}^{\infty}
    e^{-\beta N f(t)}
    \left(
        \int_{\gamma-\mathrm{i}\infty}^{\gamma+\mathrm{i}\infty}
        e^{\frac{N}{2}G(z;t)}\,dz
    \right)dt.
\end{equation}
Here $\gamma>\lambda_1$ is chosen so that the contour lies to the right of the spectrum. Consequently, the free energy can be written as
\begin{equation} \label{eq:free-energy-split}
    F_N(\beta)
    =
    \frac1N \log\left[Z_N^H(\beta)+Z_N^L(\beta)\right].
\end{equation}

The main question of the paper is to determine which of the two contributions, $Z_N^H(\beta)$ or $Z_N^L(\beta)$, gives the leading contribution to the partition function. Since the radial variable $t$ is effectively restricted by the function $f$, the comparison between the bulk and edge contributions is governed by the position of the threshold $\beta_c/\beta$ relative to the critical radial scale $\tau_{C_+}=\operatorname*{arg\,max}_{t\ge0}\bigl(tC_+-f(t)\bigr)$. Therefore, we separate the analysis according to whether
\[
\beta \tau_{C_+} < \beta_c \qquad \text{or} \qquad \beta \tau_{C_+} > \beta_c.
\]
We denote the critical inverse temperature for the soft SSK model by
\[
\beta^* := \frac{\beta_c}{\tau_{C_+}},
\]
with the convention that $\beta^*=+\infty$ when $\tau_{C_+}=0$. When $\tau_{C_+}>0$, the critical inverse temperature $\beta^*$ is finite, and at $\beta=\beta^*$ the point $t=\tau_{C_+}$ coincides with the transition point $\beta_c/\beta$ while the saddle point reaches the spectral edge. Below this threshold, the main contribution comes from the bulk saddle. Above it, the leading contribution is governed by the edge and, ultimately, by the largest eigenvalue.

\section{High-temperature regime, $\beta \in (0,\beta^*)$} \label{sec:high}
\subsection{Estimate of $Z_N^H(\beta)$ for $\beta\in(0,\beta^*)$}

In this section, we estimate the bulk contribution $Z_N^H(\beta)$. Recall that
\[
\beta_c := \frac12\int_{C_-}^{C_+}\frac{d\nu(k)}{C_+-k},
\]
and that $\tau_{C_+} = \operatorname*{arg\,max}_{t \ge 0} \bigl( tC_+ - f(t) \bigr)$. We assume $0 < \beta < \beta^*$. Under this condition, we show that the main contribution to $Z_N^H(\beta)$ comes from a saddle point located strictly to the right of the spectral edge $C_+$.

We first rewrite $Z_N^H(\beta)$ in a form suitable for the steepest descent analysis. From the contour representation obtained in Section~\ref{sec:integral}, for any $\gamma>\lambda_1$,
\[
Z_N^H(\beta) = -\frac{\mathrm{i}N}{2} \left(\frac{\pi}{\beta}\right)^{N/2-1} \int_0^{\beta_c/\beta} e^{-\beta N f(t)} \left( \int_{\gamma-\mathrm{i}\infty}^{\gamma+\mathrm{i}\infty} e^{\frac N2G(z;t)}\,dz \right)dt,
\]
where $ G(z;t) = 2\beta tz - \frac1N\sum_{i=1}^N\log(z-\lambda_i) $. Equivalently,
\[
e^{-\beta N f(t)}e^{\frac N2G(z;t)} = \exp\left\{ -\frac12\sum_{i=1}^N\log(z-\lambda_i) \right\} \exp\{N\beta(tz-f(t))\}.
\]
Since $t\in[0,\beta_c/\beta]$ is restricted to a compact interval and $\gamma>\lambda_1$, the integrand is absolutely integrable along the vertical contour. Indeed, for $z=\gamma+\mathrm{i}y$,
\[
\prod_{i=1}^N |z-\lambda_i|^{-1/2} \leq C(1+|y|)^{-N/2}
\]
for large $|y|$. Hence, by Fubini's theorem, we may interchange the $t$-integral and the $z$-integral. We obtain
\begin{equation} \label{eq:H-Fubini-form}
    Z_N^H(\beta)
    =
    -\frac{\mathrm{i}N}{2}
    \left(\frac{\pi}{\beta}\right)^{N/2-1}
    \int_{\gamma-\mathrm{i}\infty}^{\gamma+\mathrm{i}\infty}
    \exp\left\{
        -\frac12\sum_{i=1}^N\log(z-\lambda_i)
    \right\}
    R_N(z)\,dz,
\end{equation}
where
\begin{equation} \label{eq:R-N-def}
    R_N(z)
    :=
    \int_0^{\beta_c/\beta}
    \exp\{N\beta(tz-f(t))\}\,dt.
\end{equation}

For fixed real $z$, let
\[
\tau_z:=\operatorname*{arg\,max}_{t\geq0}\{tz-f(t)\},
\]
which is uniquely defined since $f$ is strongly convex, and set
\[
V(z):=z\tau_z-f(\tau_z), \qquad z\in\mathbb R.
\]
Since the integral defining $R_N(z)$ is truncated at $\beta_c/\beta$, the maximizer of its exponent is $\tau_z$ only when $\tau_z<\frac{\beta_c}{\beta}$.
In particular, at the deterministic high-temperature saddle point,
\[
\beta\tau_{\widehat\gamma_J}
=
\frac12\int_{C_-}^{C_+}\frac{d\nu(k)}{\widehat\gamma_J-k}
<
\frac12\int_{C_-}^{C_+}\frac{d\nu(k)}{C_+-k}
=
\beta_c,
\]
and therefore $\tau_{\widehat\gamma_J}<\beta_c/\beta$. By continuity of $z\mapsto\tau_z$, the same inequality holds in a fixed real neighborhood of $\widehat\gamma_J$. Only in this neighborhood do we use Laplace's method in the form
\[
R_N(z)=e^{N\beta V(z)}\times\text{(subexponential factor)}.
\]
Outside this neighborhood we keep the original truncated integral $R_N(z)$ rather than replacing it by $e^{N\beta V(z)}$.
Combining this local exponential contribution with the spectral factor in \eqref{eq:H-Fubini-form} leads to the random phase function
\[
J(z):= 2\beta V(z) -\frac1N\sum_{i=1}^N\log(z-\lambda_i), \qquad z>\lambda_1,
\]
which gives
\[
Z_N^H(\beta) \approx -\frac{\mathrm{i}N}{2} \left(\frac{\pi}{\beta}\right)^{N/2-1} \int_{\gamma-\mathrm{i}\infty}^{\gamma+\mathrm{i}\infty} \exp\left\{\frac{N}{2}J(z) \right\} \,dz
\]

The last display is only heuristic, since $V$ and $J$ are at this point defined only for real $z>C_+$. It is meant only to identify the expected exponential phase of the contour integral. After locating the real saddle point, we will extend $V$ and $J$ holomorphically to a small complex neighborhood of that point.

Note that $J$ is random since it depends on the empirical spectral measure of $M$, and hence its saddle point is also random. To locate and analyze this random saddle point, we first introduce the deterministic counterpart of $J$, obtained by replacing the empirical spectral measure $\nu_N$ with its limiting measure $\nu$:
\begin{equation} \label{eq:Jhat-def}
\widehat J(z):=
2\beta V(z)-\int_{C_-}^{C_+}\log(z-k)\,d\nu(k),
\qquad z>C_+.
\end{equation}


\begin{lemma}[Saddle point] \label{lem:critical-point-Jhat}
Assume $0 < \beta < \beta^* $. Then there exist unique critical points $\gamma_J \in (\lambda_1,\infty)$, $\widehat\gamma_J\in(C_+,\infty)$ of $J$, $\widehat J$ respectively.
\end{lemma}

\begin{proof}
We first check $\widehat{\gamma}_J$. For $x>C_+$, differentiating \eqref{eq:Jhat-def} gives
\begin{equation} \label{eq:Jhat-prime}
    \widehat J'(x)
    =
    2\beta\tau_x
    -
    \int_{C_-}^{C_+}\frac{d\nu(k)}{x-k}.
\end{equation}
By the definition of $\beta_c$,
\[
\lim_{x\downarrow C_+} \int_{C_-}^{C_+}\frac{d\nu(k)}{x-k} = \int_{C_-}^{C_+}\frac{d\nu(k)}{C_+-k} = 2\beta_c.
\]
Therefore $ \widehat J'(C_+) = 2\beta\tau_{C_+}-2\beta_c <0, $ because $\beta\tau_{C_+}<\beta_c$.

On the other hand, since $f$ is strongly convex and $f(x)/x\to\infty$ as $x\to\infty$, we have $\tau_x\to\infty$ as $x\to\infty$. Also,
\[
\int_{C_-}^{C_+}\frac{d\nu(k)}{x-k}=O(x^{-1}) \qquad \text{as } x\to\infty.
\]
Hence $\widehat J'(x)>0$ for all sufficiently large $x$.

It remains to prove uniqueness. Since $f$ is convex, the maximizer $x\mapsto\tau_x$ is nondecreasing. On the other hand, $ x\longmapsto \int_{C_-}^{C_+}\frac{d\nu(k)}{x-k}$ is strictly decreasing on $(C_+,\infty)$. Hence
\[
\widehat J'(x)
=
2\beta\tau_x-\int_{C_-}^{C_+}\frac{d\nu(k)}{x-k}
\]
is strictly increasing on $(C_+,\infty)$. Together with the sign change above, this proves the existence and uniqueness of $\widehat\gamma_J$.

The same monotonicity argument applies to the random phase. Indeed,
\[
J'(x) = 2\beta \tau_x - \frac{1}{N}\sum_{i=1}^N \frac{1}{x-\lambda_i},
\]
where the first term is nondecreasing and the second term is the negative of a strictly decreasing function. Thus $J'$ is strictly increasing on $(\lambda_1,\infty)$. Since
\[
\lim_{x\downarrow \lambda_1} J'(x) < 0, \qquad \lim_{x\to\infty} J'(x) > 0,
\]
there exists a unique $\gamma_J \in (\lambda_1,\infty)$.
\end{proof}

We note that $\widehat{\gamma}_J$ is the same quantity denoted by $\widehat{\gamma}$ in Theorem~\ref{thm:general-theorem}. We now localize the analysis near the deterministic saddle point. Choose a small constant $\eta>0$ such that
\[
\eta<\frac12(\widehat\gamma_J-C_+),
\]
and set
\[
U_\eta:=\{z\in\mathbb C: |z-\widehat\gamma_J|<\eta\}.
\]
Since
\[
f'(\tau_{\widehat\gamma_J})=\widehat\gamma_J \qquad \text{and} \qquad f''(\tau_{\widehat\gamma_J})>0,
\]
the complex inverse function theorem implies, after possibly decreasing $\eta$, that the equation
\[
f'(\tau_z)=z
\]
has a holomorphic solution $z\mapsto \tau_z$ on $U_\eta$. Consequently,
\[
V(z):=z\tau_z-f(\tau_z), \qquad z \in U_\eta
\]
is also holomorphic. Moreover,
\[
V'(z)=\tau_z, \qquad V''(z)=\frac{1}{f''(\tau_z)}.
\]

Only in this local neighborhood do we rewrite
\[
R_N(z) = e^{N\beta V(z)}I_N(z),
\]
where
\begin{equation} \label{eq:I-N-def}
    I_N(z)
    :=
    e^{-N\beta V(z)}
    \int_0^{\beta_c/\beta}
    e^{N\beta(tz-f(t))}\,dt.
\end{equation}
Thus, for $z\in U_\eta$, the integrand in \eqref{eq:H-Fubini-form} can be written in the local steepest-descent form
\[
\exp\left\{\frac N2 J(z)\right\}I_N(z),
\]
where
\begin{equation} \label{eq:J-N-def}
    J(z)
    =
    2\beta V(z)
    -
    \frac1N\sum_{i=1}^N\log(z-\lambda_i), \qquad z\in U_\eta.
\end{equation}
For convenience, we use the same notation $V$ and $J$ for their local holomorphic extensions to $U_\eta$.

Outside $U_\eta$, however, we do not use this representation. Instead, we keep the original expression as in \eqref{eq:H-Fubini-form}, where the integrand is given by $\exp\left\{-\frac12\sum_{i=1}^N\log(z-\lambda_i)\right\} R_N(z)$. This distinction is important because the holomorphic function $V(z)$ is only needed, and only defined, in a fixed neighborhood of the saddle point.

To carry out the steepest descent analysis for the random integral, we need to compare $\gamma_J$ with the deterministic point $\widehat\gamma_J$ and to control the derivatives of $J$ in a neighborhood of the saddle point. The next lemma records the estimates needed to stabilize the random saddle.

\begin{lemma} \label{lem:J-derivative-estimates}
Let $0 < \beta < \beta^* $. Then the following statements hold with overwhelming probability.

\begin{enumerate}
    \item[(i)] Fix $\delta>0$. Uniformly for real $x\ge C_+ +\delta$,
    \[
J'(x)-\widehat J'(x)\prec N^{-1}.
    \]

    \item[(ii)] Let $U_\eta$ be the neighborhood of $\widehat\gamma_J$ on which $\tau_z$ and $V(z)=z\tau_z-f(\tau_z)$ are holomorphic.
For each fixed integer $\ell\ge 2$ and each compact set $K\subset U_\eta$, we have
    \[
J^{(\ell)}(z)=O(1)
    \]
uniformly for $z\in K$.
\end{enumerate}
\end{lemma}

\begin{proof}
Part (i) follows directly from Corollary~\ref{cor:bulk-approximation}, and part (ii) follows from holomorphicity of $V$ and the positive distance between $U_\eta$ and the spectrum.  Details are given in Appendix~\ref{app:high-temp-technical}.
\end{proof}

We now compare the random critical point $\gamma_J$ with the deterministic critical point $\widehat\gamma_J$. Since we are in the high-temperature regime, $\widehat\gamma_J$ is separated from the spectral edge $C_+$. Hence $\widehat J''(\widehat\gamma_J)>0$, and the critical point is stable under small perturbations of the phase function. Using the estimate $J'-\widehat J'\prec N^{-1}$ from Lemma~\ref{lem:J-derivative-estimates}, we obtain the following approximation.

\begin{lemma} \label{lem:gamma-approximation}
Let $0 < \beta < \beta^*$. Then
    \begin{equation} \label{eq:gamma-J-approximation}
        |\gamma_J - \widehat{\gamma}_J| \prec N^{-1},
    \end{equation}
with overwhelming probability. In particular, there exists a constant $c>0$ independent of $N$ such that
    \begin{equation} \label{eq:gamma-J-edge-gap}
        \gamma_J - \lambda_1 > c
    \end{equation}
with overwhelming probability.
\end{lemma}

\begin{proof}
The estimate follows from the nondegeneracy of the deterministic saddle, Lemma~\ref{lem:J-derivative-estimates}(i), and the monotonicity of $J'$ on $(\lambda_1,\infty)$.  The uniform separation from the spectrum then follows from $\widehat\gamma_J>C_+$ and $\lambda_1\to C_+$.  See Appendix~\ref{app:high-temp-technical} for details.
\end{proof}

The first step in estimating the contour integral is to factor out $e^{\frac{N}{2}J(\gamma_J)}$ and estimate the remaining integral
\[
\int e^{\frac{N}{2}(J(z)-J(\gamma_J))}\,dz
\]
by the steepest descent method.  Since the quadratic term at the critical point gives the leading contribution, we need to compare $J(\gamma_J)$ and $J''(\gamma_J)$ with the corresponding quantities at the deterministic saddle point $\widehat\gamma_J$.  The following lemma provides the required estimates.

\begin{lemma}\label{lem:J-saddle-comparison}
Let $0 < \beta < \beta^*$. Then
\[
J(\gamma_J) = J(\widehat\gamma_J)+O_\prec(N^{-2}), \qquad J''(\gamma_J) = J''(\widehat\gamma_J)+O_\prec(N^{-1})
\]
with overwhelming probability.
\end{lemma}

\begin{proof}
This is an immediate Taylor expansion around $\widehat\gamma_J$, using Lemma~\ref{lem:gamma-approximation} and the uniform derivative bounds from Lemma~\ref{lem:J-derivative-estimates}.  See Appendix~\ref{app:high-temp-technical}.
\end{proof}

We next estimate the additional factor $I_N(z)$ defined by \eqref{eq:I-N-def}. In the local steepest-descent analysis near $\gamma_J$, the factor $e^{\frac N2 J(z)}$ gives the main exponential contribution, while $I_N(z)$ contributes an additional prefactor. The following lemma shows that, uniformly on the scale $z=\gamma_J+O(N^{-1/2+\epsilon})$, this prefactor is of order $N^{-1/2}$.

\begin{lemma} \label{lem:I-estimate}
Let $0 < \beta < \beta^*$. For any sufficiently small $\epsilon>0$, we have
\[
I_N\left(\gamma_J+\frac{ic}{\sqrt N}\right) = \sqrt{\frac{2\pi}{\beta N f''(\tau_{\gamma_J})}} \left(1+O(N^{-1/4+C\epsilon})\right)
\]
uniformly for $|c|\le N^\epsilon$, with overwhelming probability.
\end{lemma}

\begin{proof}
Let $ z=\gamma_J+\frac{ic}{\sqrt N}, $ for $ |c|\le N^\epsilon.$ We prove the estimate on the overwhelming probability event where Lemma~\ref{lem:gamma-approximation} holds.
Since $0 < \beta < \beta^*$, the deterministic critical point $\widehat{\gamma}_J$ satisfies $\widehat{\gamma}_J>C_+$. Thus
\[
\beta\tau_{\widehat{\gamma}_J} = \frac12\int_{C_-}^{C_+}\frac{d\nu(x)}{\widehat{\gamma}_J-x} < \frac12\int_{C_-}^{C_+}\frac{d\nu(x)}{C_+-x} = \beta_c.
\]
Hence $ 0<\tau_{\widehat{\gamma}_J}<\frac{\beta_c}{\beta}. $ Since $\gamma_J-\widehat{\gamma}_J=O_{\prec}(N^{-1})$, it follows that $\tau_{\gamma_J}$ remains in a fixed compact subinterval of $(0,\beta_c/\beta)$ with overwhelming probability.

By the analytic inverse function theorem, the map $z\mapsto \tau_z$ is holomorphic in a fixed neighborhood of $\gamma_J$. Therefore, uniformly for $|c|\le N^\epsilon$,
\begin{equation} \label{eq:tau-z-expansion}
    \tau_z
    =
    \tau_{\gamma_J}
    +
    \frac{ic}{f''(\tau_{\gamma_J})\sqrt N}
    +
    O(N^{-1+2\epsilon}).
\end{equation}

Let $ \delta=N^{-5/12}. $ Then since $\epsilon>0$ is sufficiently small, note that $|\tau_z-\tau_{\gamma_J}|=o(\delta)$. We now decompose
\[
I_N(z)=I_1+I_2,
\]
where
\[
I_1 = \int_{\tau_{\gamma_J}-\delta}^{\tau_{\gamma_J}+\delta} e^{N\beta\{(t-\tau_z)z-(f(t)-f(\tau_z))\}}\,dt
\]
and
\[
I_2 = \int_{[0,\beta_c/\beta]\setminus [\tau_{\gamma_J}-\delta,\tau_{\gamma_J}+\delta]} e^{N\beta\{(t-\tau_z)z-(f(t)-f(\tau_z))\}}\,dt .
\]

We first estimate $I_1$. Since $f'(\tau_z)=z$, Taylor expansion at $\tau_z$ gives
\[
(t-\tau_z)z-(f(t)-f(\tau_z)) = -\frac{f''(\tau_z)}{2}(t-\tau_z)^2 +O((t-\tau_z)^3),
\]
uniformly for $t\in[\tau_{\gamma_J}-\delta,\tau_{\gamma_J}+\delta]$. Setting $s=t-\tau_z$, we obtain
\[
I_1 = \int_{\tau_{\gamma_J}-\tau_z-\delta}^{\tau_{\gamma_J}-\tau_z+\delta} \exp\left\{ -N\beta\left( \frac{f''(\tau_z)}{2}s^2+O(s^3) \right) \right\}\,ds .
\]
Since the original integrand is holomorphic near the saddle point, we may first deform the $s$-contour to the real interval $[-\delta,\delta]$ and then apply the above Taylor expansion. By \eqref{eq:tau-z-expansion}, the connecting segments have length $O(N^{-1/2+\epsilon})$, and the Taylor expansion at $\tau_z$ shows that their contribution is $O(e^{-cN^{1/6}})$ for sufficiently small $\epsilon>0$. Hence
\[
I_1
=
\int_{-\delta}^{\delta}
\exp\left\{
-N\beta\left(
\frac{f''(\tau_z)}{2}s^2+O(s^3)
\right)
\right\}ds
+O(e^{-cN^{1/6}}).
\]
Since
\[
f''(\tau_z)
=
f''(\tau_{\gamma_J})+O(N^{-1/2+\epsilon})
\]
and $N|s|^3=O(N^{-1/4})$ for $|s|\le\delta=N^{-5/12}$, we obtain
\[
I_1
=
\left(1+O(N^{-1/4+C\epsilon})\right)
\int_{-\delta}^{\delta}
\exp\left\{
-\frac{N\beta f''(\tau_{\gamma_J})}{2}s^2
\right\}ds.
\]
Note that $N\delta^2=N^{1/6}$, so the Gaussian tails are exponentially small. Therefore,
\begin{equation} \label{eq:I1-estimate}
I_1
=
\sqrt{\frac{2\pi}{\beta Nf''(\tau_{\gamma_J})}}
\left(1+O(N^{-1/4+C\epsilon})\right).    
\end{equation}

It remains to estimate $I_2$. Let
\begin{equation} \label{eq:psi}
    \psi(t):=t\gamma_J-f(t).
\end{equation}
Since $f'(\tau_{\gamma_J})=\gamma_J$, the function $\psi$ has its unique maximum at $t=\tau_{\gamma_J}$. Moreover, by the strong convexity of $f$, there exists a constant $c>0$, independent of $N$, such that
\begin{equation} \label{eq:psi-property}
    \psi(t)\le \psi(\tau_{\gamma_J})-c(t-\tau_{\gamma_J})^2 = V(\gamma_J) -c(t-\tau_{\gamma_J})^2
\end{equation}
for all $t\in[0,\beta_c/\beta]$. Hence, on the domain of $I_2$,
\[
\psi(t)\le V(\gamma_J)-c\delta^2.
\]
Since $\re z=\gamma_J$, we have
\begin{align*}
    |I_2|
    &\le
    \exp\{-N\beta\re V(z)\}
    \int_{[0,\beta_c/\beta]\setminus
    [\tau_{\gamma_J}-\delta,\tau_{\gamma_J}+\delta]}
    e^{N\beta(t\gamma_J-f(t))}\,dt  \\
    &\le
    C\exp\left\{
        N\beta\bigl(V(\gamma_J)-\re V(z)\bigr)
        -cN\delta^2
    \right\}.
\end{align*}
Since $V'(z)=\tau_z$, Taylor expansion around $\gamma_J$ gives
\[
V(z) = V(\gamma_J) + \tau_{\gamma_J}(z-\gamma_J) + \frac{(z-\gamma_J)^2}{2f''(\tau_{\gamma_J})} + O(|z-\gamma_J|^3).
\]
Therefore, $ V(\gamma_J)-\re V(z) = O(N^{-1+2\epsilon}), $ uniformly for $|c|\le N^\epsilon$. Since $N\delta^2=N^{1/6}$, we obtain
\[
|I_2| \le C\exp\left\{ CN^{2\epsilon}-cN^{1/6} \right\} = o(N^{-1/2})
\]
for sufficiently small $\epsilon>0$.

Combining \eqref{eq:I1-estimate} with the estimate for $I_2$, we conclude that
\[
I_N\left(\gamma_J+\frac{ic}{\sqrt N}\right) = \sqrt{\frac{2\pi}{\beta N f''(\tau_{\gamma_J})}} \left(1+O(N^{-1/4+C\epsilon})\right),
\]
uniformly for $|c|\le N^\epsilon$, with overwhelming probability. This proves the lemma.
\end{proof}

Now, we estimate $Z_N^H(\beta)$ as follows.
\begin{lemma} \label{lem:high-temp-approximation}
Let $0 < \beta < \beta^*$. Then
\begin{equation} \label{eq:high-temp-approximation}
    \int_{\gamma_J-\mathrm{i}\infty}^{\gamma_J+\mathrm{i}\infty}
    \exp\left\{
        -\frac12\sum_{i=1}^N\log(z-\lambda_i)
    \right\}
    R_N(z)\,dz
    =
    \mathrm{i} e^{\frac{N}{2}J(\widehat\gamma_J)}
    \frac{2\sqrt{2}\pi}
    {N\sqrt{\beta f''(\tau_{\widehat\gamma_J})J''(\widehat\gamma_J)}}
    \left(1+O(N^{-1/4+C\epsilon})\right)
\end{equation}
with overwhelming probability.
\end{lemma}\begin{proof}
Near $\gamma_J$, set $z=\gamma_J+\mathrm{i}y/\sqrt N$.  Lemma~\ref{lem:I-estimate} and the quadratic expansion of $J$ give the Gaussian local contribution. The part $|y|>N^\epsilon$ is negligible by the uniform separation in \eqref{eq:gamma-J-edge-gap}. Finally, Lemma~\ref{lem:J-saddle-comparison} and Corollary~\ref{cor:bulk-approximation} replace the random saddle data by their deterministic counterparts. The complete estimate, including the contour-tail bound, is given in Appendix~\ref{app:high-temp-technical}.
\end{proof}
\subsection{Estimate of $Z_N^L(\beta)$ for $0 < \beta < \beta^*$}
\label{sec:LN-high-temp}

We now show that the contribution $Z_N^L(\beta)$ is negligible in the high-temperature regime. Recall that $Z_N^L(\beta)$ is the part of the partition function corresponding to the radial region $t\geq \frac{\beta_c}{\beta}$. In the regime $0 < \beta < \beta^*$, the saddle point of the bulk contribution lies strictly inside the interval $\left(0,\frac{\beta_c}{\beta}\right)$. Thus the integral defining $Z_N^L(\beta)$ starts strictly beyond the radial saddle, and we expect it to be exponentially smaller than $Z_N^H(\beta)$. When estimating $Z_N^L(\beta)$ we will take $\gamma = \gamma_J$.

We first observe that
\begin{equation} \label{eq:gamma-less-than-edge-radial-derivative}
    \gamma_J < f'(\frac{\beta_c}{\beta})
\end{equation}
with overwhelming probability. Indeed, $f'(\frac{\beta_c}{\beta})>C_+$ since $\frac{\beta_c}{\beta}>\tau_{C_+}$. Moreover,
\[
\widehat J'(f'(\frac{\beta_c}{\beta})) = 2\beta \frac{\beta_c}{\beta} - \int_{C_-}^{C_+}\frac{d\nu(k)}{f'(\frac{\beta_c}{\beta})-k} = 2\beta_c - \int_{C_-}^{C_+}\frac{d\nu(k)}{f'(\frac{\beta_c}{\beta})-k} >0,
\]
because
\[
\int_{C_-}^{C_+}\frac{d\nu(k)}{f'(\frac{\beta_c}{\beta})-k} < \int_{C_-}^{C_+}\frac{d\nu(k)}{C_+-k} = 2\beta_c.
\]
Since $\widehat J'$ is strictly increasing and $\widehat J'(\widehat\gamma_J)=0$, it follows that $ \widehat\gamma_J<f'(\frac{\beta_c}{\beta}). $ By Lemma \ref{lem:gamma-approximation}, $\gamma_J=\widehat\gamma_J+O_{\prec}(N^{-1})$, and hence \eqref{eq:gamma-less-than-edge-radial-derivative} holds with overwhelming probability. In particular, there exists a constant $c>0$, independent of $N$, such that
\begin{equation} \label{eq:positive-gap-radial-tail}
    f'(\frac{\beta_c}{\beta})-\gamma_J\geq c
\end{equation}
with overwhelming probability.

Now using Fubini's theorem, we may write $Z_N^L(\beta)$ as
\begin{align}
    Z_N^L(\beta)
    &=
    -\frac{\mathrm{i}N}{2}
    \left(\frac{\pi}{\beta}\right)^{N/2-1}
    \int_{\gamma_J-\mathrm{i}\infty}^{\gamma_J+\mathrm{i}\infty}
    \exp\left\{
        N\beta(\frac{\beta_c}{\beta} z-f(\frac{\beta_c}{\beta}))
        -
        \frac12\sum_{i=1}^N\log(z-\lambda_i)
    \right\}
    W_N(z)\,dz
\end{align}
where
\[
W_N(z) := \int_{\frac{\beta_c}{\beta}}^{\infty} \exp\left\{ N\beta\bigl[(t-\frac{\beta_c}{\beta})z-(f(t)-f(\frac{\beta_c}{\beta}))\bigr] \right\}\,dt .
\]
Indeed, strong convexity of $f$ gives $ f(t)-f(\frac{\beta_c}{\beta}) \geq f'(\frac{\beta_c}{\beta})(t-\frac{\beta_c}{\beta}), $ when $t\geq \frac{\beta_c}{\beta}$. Therefore, using \eqref{eq:positive-gap-radial-tail},
\begin{align}
    |W_N(\gamma_J+\mathrm{i}y)|
    &\leq
    \int_{\frac{\beta_c}{\beta}}^{\infty}
    \exp\left\{
        N\beta(t-\frac{\beta_c}{\beta})(\gamma_J-f'(\frac{\beta_c}{\beta}))
    \right\}\,dt  \notag \\
    &=
    \int_0^{\infty}
    \exp\left\{
        -N\beta(f'(\frac{\beta_c}{\beta})-\gamma_J)u
    \right\}\,du  \notag \\
    &\leq
    \frac{C}{N}.
    \label{eq:radial-tail-bound}
\end{align}

Taking absolute value of $Z_N^L(\beta)$ and applying \eqref{eq:radial-tail-bound}, we obtain
\begin{align}
    |Z_N^L(\beta)|
    &\leq
    C
    \left(\frac{\pi}{\beta}\right)^{N/2-1}
    \exp\left\{
        N\beta(\frac{\beta_c}{\beta}\gamma_J-f(\frac{\beta_c}{\beta}))
        -
        \frac12\sum_{i=1}^N\log(\gamma_J-\lambda_i)
    \right\}                                      \notag \\
    &\quad\times
    \int_{-\infty}^{\infty}
    \exp\left\{
        -\frac14\sum_{i=1}^N
        \log\left(
            1+\frac{s^2}{(\gamma_J-\lambda_i)^2}
        \right)
    \right\}\,ds .
    \label{eq:LN-bound-before-s-integral}
\end{align}
By Lemma~\ref{lem:gamma-approximation}, $\gamma_J$ stays in a fixed compact subset of $(C_+,\infty)$ with overwhelming probability. Thus
\begin{equation} \label{eq:vertical-tail-integral-bound}
    \int_{-\infty}^{\infty}
    \exp\left\{
        -\frac14\sum_{i=1}^N
        \log\left(
            1+\frac{s^2}{(\gamma_J-\lambda_i)^2}
        \right)
    \right\}\,ds
    \leq
    \int_{-\infty}^{\infty}(1+cs^2)^{-N/4}\,ds
    \leq C.
\end{equation}
Combining \eqref{eq:LN-bound-before-s-integral} and \eqref{eq:vertical-tail-integral-bound} gives
\begin{equation} \label{eq:LN-exp-bound}
    |Z_N^L(\beta)|
    \leq
    C
    \left(\frac{\pi}{\beta}\right)^{N/2-1}
    \exp\left\{
        N\beta(\frac{\beta_c}{\beta}\gamma_J-f(\frac{\beta_c}{\beta}))
        -
        \frac12\sum_{i=1}^N\log(\gamma_J-\lambda_i)
    \right\}
\end{equation}
with overwhelming probability.

We can compare the exponent in \eqref{eq:LN-exp-bound} with the exponent of $Z_N^H(\beta)$. Consider $\psi : [0,\infty) \to \R$ as defined in \eqref{eq:psi}.
The function $\psi$ is strictly concave and has its unique maximizer $\tau_{\gamma_J}$. Moreover, by \eqref{eq:gamma-less-than-edge-radial-derivative}, we have $ \tau_{\gamma_J}<\frac{\beta_c}{\beta} $ with overwhelming probability. Then there exists a constant $c>0$, independent of $N$, such that
\begin{equation} \label{eq:radial-exponential-gap}
    \psi(\frac{\beta_c}{\beta})
    \leq
    \psi(\tau_{\gamma_J})-c.
\end{equation}
Equivalently, $ \frac{\beta_c}{\beta}\gamma_J-f(\frac{\beta_c}{\beta}) \leq V(\gamma_J)-c. $ Thus \eqref{eq:LN-exp-bound} yields
\begin{equation} \label{eq:LN-final-exp-bound}
    |Z_N^L(\beta)|
    \leq
    C
    \left(\frac{\pi}{\beta}\right)^{N/2-1}
    \exp\left\{
        \frac{N}{2}J(\gamma_J)-cN
    \right\}
\end{equation}
with overwhelming probability.

On the other hand, Lemma \ref{lem:high-temp-approximation} gives the leading asymptotics of $Z_N^H(\beta)$:
\[
Z_N^H(\beta) = \left(\frac{\pi}{\beta}\right)^{N/2-1} \exp\left\{ \frac{N}{2}J(\widehat\gamma_J) + O_{\prec}(\log N) \right\}.
\]
Since $ J(\gamma_J)=J(\widehat\gamma_J)+O_{\prec}(N^{-2})$, we conclude from \eqref{eq:LN-final-exp-bound} that
\[
\frac{|Z_N^L(\beta)|}{|Z_N^H(\beta)|} \leq e^{-cN}
\]
with overwhelming probability, after decreasing $c>0$ if necessary. Hence
\[
Z_N^L(\beta)=o(Z_N^H(\beta))
\]
with overwhelming probability. Therefore, in the high-temperature regime $0 < \beta < \beta^*$, the free energy is determined by the bulk contribution $Z_N^H(\beta)$.

\section{Low-temperature regime, $\beta \in (\beta^*,\infty)$} \label{sec:low}
\subsection{Estimate of $Z_N^L(\beta)$ for $\beta \in (\beta^*,\infty)$}
Throughout the low-temperature analysis, we assume that $\tau_{C_+}>0$. In the low-temperature regime, the radial saddle lies in the region $(\beta_c/\beta,\infty)$. For such $t$, the contour integral in \eqref{eq:L-N} behaves as in the low-temperature phase of the SSK model with an effective inverse temperature $\beta t$. Consequently, the saddle point of the phase $G(z;t)$ is expected to be pinned near the largest eigenvalue $\lambda_1$, rather than near a deterministic point outside the limiting support.

We begin by making this observation quantitative. The next lemma shows that, uniformly for $t$ separated from $\beta_c/\beta$, the saddle point $\gamma_G(t)$ satisfies $ \gamma_G(t)-\lambda_1 \prec N^{-1}. $ This estimate will be used to expand $G(\gamma_G(t);t)$ in terms of the largest eigenvalue and the deterministic edge $C_+$.

Choose $\delta>0$ sufficiently small and then choose $T>0$ sufficiently large so that
\begin{equation} \label{eq:delta-T-low-temp}
    \frac{\beta_c}{\beta}+\delta<\tau_{C_+}<T.
\end{equation}

\begin{lemma} \label{lem:critical-point-G}
Suppose that $  \frac{\beta_c}{\beta}+\delta<t<T$. Let $\gamma_G(t)$ be the critical point of
\[
G(z;t) = 2\beta t z-\frac{1}{N}\sum_{i=1}^N \log(z-\lambda_i), \qquad z>\lambda_1.
\]
Then
\begin{equation} \label{eq:gamma-lambda1-bound}
    \frac{1}{3\beta NT}
    \leq
    \gamma_G(t)-\lambda_1
    \prec
    N^{-1}
\end{equation}
uniformly with overwhelming probability.
\end{lemma}

\begin{proof}
For fixed $t$, this is the standard low-temperature saddle estimate for the spherical model with effective inverse temperature $\beta t$.  The bounds are uniform for $t\in[\beta_c/\beta+\delta,T]$; see Appendix~\ref{app:low-temp-technical}.
\end{proof}

Having located the saddle point near the top eigenvalue, we next estimate the value of the phase at the saddle point and its higher derivatives.  The estimate of $G(\gamma_G(t);t)$ separates the random edge contribution, which depends on $\lambda_1$, from the deterministic logarithmic term coming from the limiting spectral measure.  The bounds on the derivatives are also needed to control the local behavior of $G$ along the steepest descent contour.

\begin{lemma} \label{lem:G-approximation}
Suppose  $T > t > \frac{\beta_c}{\beta} + \delta$. Let $\gamma_G(t)$ be defined as in Lemma~\ref{lem:critical-point-G}. Then
\begin{equation} \label{eq:G-gamma-approx-low}
    G(\gamma_G(t);t)
    =
    2\beta t\lambda_1
    -
    \int_{C_-}^{C_+}\log(C_+-x)\,d\nu(x)
    -
    2\beta_c(\lambda_1-C_+)
    +
    O_\prec(N^{-1}),
\end{equation}
uniformly with overwhelming probability. Moreover, for every $0<\epsilon<\frac{1}{12}$, there exists a constant $C_0>0$, depending only on $\beta$, $\delta$, and $T$, such that
\begin{equation} \label{eq:G-derivative-bound-low}
    N^{l-1-4l\epsilon}
    \leq
    \frac{(-1)^l}{(l-1)!}G^{(l)}(\gamma_G(t);t)
    \leq
    C_0^l N^{l-1+3\epsilon}
\end{equation}
uniformly for all $l=2,3,\ldots$, with overwhelming probability.
\end{lemma}

\begin{proof}
The phase expansion follows from Corollary~\ref{cor:edge-approximation} and Lemma~\ref{lem:critical-point-G}.  The derivative bounds are the usual edge saddle estimates.  Details are given in Appendix~\ref{app:low-temp-technical}.
\end{proof}

\begin{lemma} \label{lem:estimate-exp-G}
Suppose $T > t > \frac{\beta_c}{\beta} + \delta$. Let $\gamma_G(t)$ be defined as in Lemma~\ref{lem:critical-point-G}. Then there exists a positive random factor $K_N(t)$ satisfying $ N^{-C}<K_N(t)<C $ for some constant $C>0$, independent of $N$, such that
\begin{equation} \label{eq:estimate-exp-G}
    \int_{\gamma_G(t)-\mathrm{i}\infty}^{\gamma_G(t)+\mathrm{i}\infty}
    e^{\frac{N}{2}G(z;t)}\,dz
    =
    \mathrm{i} e^{\frac{N}{2}G(\gamma_G(t);t)}K_N(t)
\end{equation}
uniformly with overwhelming probability. Moreover, for all $t>\frac{\beta_c}{\beta}$, the same formula holds with $0 < K_N(t)< C$ with overwhelming probability, uniformly in $t$.
\end{lemma}

\begin{proof}
This is the uniform low-temperature steepest-descent estimate for the spherical contour integral with effective inverse temperature $\beta t$. The construction of the contour and the upper and lower bounds for $K_N(t)$ are collected in Appendix~\ref{app:low-temp-technical}.
\end{proof}

Now we estimate $Z_N^L(\beta)$. Note that
\begin{equation} \label{eq:tau-lambda1-inside}
    \frac{\beta_c}{\beta}+\delta<\tau_{\lambda_1}<T
\end{equation}
with overwhelming probability since $f$ is strongly convex and $\lambda_1\to C_+$ with overwhelming probability. We decompose $Z_N^L(\beta)$ by
\[
Z_N^L(\beta) = Z_N^{L(1)}(\beta)+Z_N^{L(2)}(\beta)+Z_N^{L(3)}(\beta),
\]
where the three terms correspond to the intervals $ I_1 := \left[\frac{\beta_c}{\beta},\frac{\beta_c}{\beta}+\delta\right], \; I_2 :=\left[\frac{\beta_c}{\beta}+\delta,T\right], \; I_3 := [T,\infty), $ respectively, and
\begin{align} \label{eq:LN-main-before-G-approx}
    Z_N^{L(i)}(\beta)
    &:=
    \frac{N}{2}
    \left(\frac{\pi}{\beta}\right)^{N/2-1}
    \int_{I_i}
    K_N(t)
    e^{-\beta N f(t)}
    e^{\frac{N}{2}G(\gamma_G(t);t)}\,dt .
\end{align}
The equality follows from Lemma \ref{lem:estimate-exp-G}.

We first estimate the main part $Z_N^{L(2)}(\beta)$. From Lemma \ref{lem:G-approximation}, we obtain
\begin{align} \label{eq:LN-main-after-G-approx}
    Z_N^{L(2)}(\beta)
    &=
    \frac{N}{2}
    \left(\frac{\pi}{\beta}\right)^{N/2-1}
    \exp\left\{
        -\frac{N}{2}
        \int_{C_-}^{C_+}\log(C_+-x)\,d\nu(x)
        -N\beta_c(\lambda_1-C_+)
        +O_\prec(1)
    \right\}
    \nonumber \\
    &\quad \times
    \int_{I_2}
    K_N(t)
    \exp\left\{
        \beta N(t\lambda_1-f(t))
    \right\}\,dt
\end{align}
uniformly for $t\in[\beta_c/\beta+\delta,T]$. Since $f$ is strongly convex, $t \mapsto t\lambda_1-f(t)$ is strictly concave and has its unique maximum at $t=\tau_{\lambda_1}$. This implies that the maximum lies in the interval $I_2$ with overwhelming probability by \eqref{eq:tau-lambda1-inside}. {
Note that the low-temperature assumption in Theorem~\ref{thm:general-theorem} gives a holomorphic extension of $f$ near $\tau_{C_+}$. Since $\lambda_1\to C_+$ with overwhelming probability, the same extension contains $\tau_{\lambda_1}$ for all sufficiently large $N$ with overwhelming probability.
}
Therefore, Laplace's method gives
\begin{equation} \label{eq:radial-laplace-low}
    \int_{I_2}
    e^{\beta N(t\lambda_1-f(t))}\,dt
    =
    e^{\beta N V(\lambda_1)}
    \sqrt{\frac{2\pi}{\beta N f''(\tau_{\lambda_1})}}
    \left(1+O(N^{-1})\right),
\end{equation}
where $ V(\lambda_1) = \lambda_1\tau_{\lambda_1}-f(\tau_{\lambda_1}). $ Since $K_N(t)$ is bounded above and below by polynomial powers of $N$ uniformly in $t \in I_2$, \eqref{eq:radial-laplace-low} implies that there exists a random factor $\widetilde K_N$ satisfying $ N^{-C}<\widetilde K_N<C $ such that
\begin{align}
    Z_N^{L(2)}(\beta)
    &=
    \frac{\widetilde K_N N}{2}
    \left(\frac{\pi}{\beta}\right)^{N/2-1}
    \exp\left\{
        -\frac{N}{2}
        \int_{C_-}^{C_+}\log(C_+-x)\,d\nu(x)
        -N\beta_c(\lambda_1-C_+) + O_\prec(1)
    \right\}
    \nonumber \\
    &\quad \times
    e^{\beta N V(\lambda_1)}
    \sqrt{\frac{2\pi}{\beta N f''(\tau_{\lambda_1})}}
    \label{eq:LN-main-final}
\end{align}

It remains to check that the two complementary parts $Z_N^{L(1)}$ and $Z_N^{L(3)}$ are negligible. First note that, by \eqref{eq:tau-lambda1-inside}, the intervals $I_1$ and $I_3$ stay a positive distance away from the maximizer $\tau_{\lambda_1}$ with overwhelming probability. Let
\[
\Phi_N(t) := \beta(\gamma_G(t)t - f(t)) - \frac{1}{2N}\sum_{i=1}^N \log(\gamma_G(t) - \lambda_i)
\]
for $t > \frac{\beta_c}{\beta}$. Then
\[
\Phi_N'(t) = \beta(\gamma_G(t)-f'(t)).
\]
$\Phi_N'(t)$ is decreasing because $\gamma_G(t)$ is decreasing and, by strong convexity of $f$, $f'$ is increasing. Now let $t_*$ be the critical point of $\Phi_N$. Then $t_*$ satisfies
\[
\gamma_G(t_*)=f'(t_*).
\]
By the definition of $\tau_z$, this implies
\[
t_*=\tau_{\gamma_G(t_*)}.
\]
We claim that $t_*$ stays inside $(\beta_c/\beta+\delta,T)$ with overwhelming probability.
{
Choose a deterministic constant $\eta>0$ such that
\[
[\tau_{C_+}-2\eta,\tau_{C_+}+2\eta]
\subset
(\beta_c/\beta+\delta,T).
\]
Since $\lambda_1\to C_+$ and $z\mapsto\tau_z$ is continuous, with overwhelming probability
\[
[\tau_{\lambda_1}-\eta,\tau_{\lambda_1}+\eta]
\subset
(\beta_c/\beta+\delta,T).
\]
}
By strong convexity of $f$, there exists $m>0$ such that
\[
f'(\tau_{\lambda_1}+\eta)\geq\lambda_1+m\eta,
\qquad
f'(\tau_{\lambda_1}-\eta)\leq\lambda_1-m\eta.
\]
By Lemma \ref{lem:critical-point-G},
\[
\gamma_G(\tau_{\lambda_1} \pm \eta) = \lambda_1 + O_\prec(N^{-1}).
\]
Hence, we have
\[
\Phi_N'(\tau_{\lambda_1} - \eta) \geq \beta(\lambda_1 + O_\prec(N^{-1}) - \lambda_1 + m\eta) > 0
\]
with overwhelming probability. Similarly,
\[
\Phi_N'(\tau_{\lambda_1} + \eta) \leq \beta(\lambda_1 + O_\prec(N^{-1}) - \lambda_1 - m\eta) < 0
\]
with overwhelming probability. Since $\Phi_N'(t)$ is a decreasing function, the claim holds.

{
We next compare the size of the main contribution $Z_N^{L(2)}$ with the value $e^{N\Phi_N(t_*)}$. Since $t_*\in I_2$ with overwhelming probability, Lemma~\ref{lem:critical-point-G} gives
\[
\gamma_G(t_*)=\lambda_1+O_\prec(N^{-1}).
\]
Using $\gamma_G(t_*)=f'(t_*)$ and $f'(\tau_{\lambda_1})=\lambda_1$, strong convexity of $f$ implies
\[
|t_*-\tau_{\lambda_1}|=O_\prec(N^{-1}).
\]
Moreover, Lemma~\ref{lem:G-approximation} yields, uniformly on $I_2$,
\[
\Phi_N(t)
=
-\frac12\int_{C_-}^{C_+}\log(C_+-x)\,d\nu(x)
-\beta_c(\lambda_1-C_+)
+\beta\bigl(t\lambda_1-f(t)\bigr)
+O_\prec(N^{-1}).
\]
Since $t\mapsto t\lambda_1-f(t)$ has its maximum at $\tau_{\lambda_1}$ and
$t_*-\tau_{\lambda_1}=O_\prec(N^{-1})$, we obtain
\[
\Phi_N(t_*)
=
-\frac12\int_{C_-}^{C_+}\log(C_+-x)\,d\nu(x)
-\beta_c(\lambda_1-C_+)
+\beta V(\lambda_1)
+O_\prec(N^{-1}).
\]
Combining this identity with \eqref{eq:LN-main-final}, the lower bound
$\widetilde K_N\geq N^{-C}$, and the fact that $f''(\tau_{\lambda_1})$ stays bounded above and below by positive constants with overwhelming probability, we obtain the following slightly weaker bound. For any fixed $0<\rho<1$,
\begin{equation}\label{eq:ZL2-lower-Phi}
Z_N^{L(2)}(\beta)
\geq
\frac{N}{2}\left(\frac{\pi}{\beta}\right)^{N/2-1}
N^{-C}\exp\{N\Phi_N(t_*)-N^\rho\}
\end{equation}
with overwhelming probability, after enlarging $C$ if necessary. Here the factor $e^{-N^\rho}$ absorbs the $O_\prec(1)$ term in the exponent of \eqref{eq:LN-main-final}.
}

Therefore there exists $c>0$ such that
\[
\Phi_N(\beta_c/\beta + \delta) \leq \Phi_N(t_*) - c
\]
with overwhelming probability. Thus, using that $0 < K_N(t) \leq C$ uniformly, we have
{
\[
\begin{aligned}
Z_N^{L(1)}(\beta)
&=
\frac{N}{2}\left(\frac{\pi}{\beta}\right)^{N/2-1}
\int_{I_1}K_N(t)e^{N\Phi_N(t)}\,dt \\
&\leq
\frac{CN}{2}\left(\frac{\pi}{\beta}\right)^{N/2-1}
|I_1|e^{N\Phi_N(t_*)-cN}.
\end{aligned}
\]
Comparing this with \eqref{eq:ZL2-lower-Phi}, we get
\[
\frac{Z_N^{L(1)}(\beta)}{Z_N^{L(2)}(\beta)}
\leq N^C e^{-cN+N^\rho}
=
O(e^{-c'N})
\]
for some $c'>0$, with overwhelming probability.
}
Now we estimate $Z_N^{L(3)}$. By the strong convexity of $f$ and the condition $f(x)/x \to \infty$ as $x \to \infty$, there exists $c_1, c_2 > 0$ such that
\[
\Phi_N(t) \leq \Phi_N(t_*) - c_1 - c_2(t-T), \quad t \geq T
\]
for some large constant $T$. Therefore, we have
\[
\begin{aligned}
Z_N^{L(3)}(\beta) &= \frac{N}{2} \left( \frac{\pi}{\beta} \right)^{N/2 -1} \int_{I_3} K_N(t) e^{N\Phi_N(t)} \, dt \\ &\leq \frac{N}{2} \left( \frac{\pi}{\beta} \right)^{N/2 -1} Ce^{N\Phi_N(t_*) - c_1N} \int_{I_3} K_N(t) e^{-N c_2 (t - T)} \, dt \\ &\leq \frac{N}{2} \left( \frac{\pi}{\beta} \right)^{N/2 -1} Ce^{N\Phi_N(t_*) - c_1N} \frac{1}{Nc_2}.
\end{aligned}
\]
{
Comparing the preceding upper bound with \eqref{eq:ZL2-lower-Phi}, we similarly obtain
\[
\frac{Z_N^{L(3)}(\beta)}{Z_N^{L(2)}(\beta)}
\leq
N^C e^{-c_1N+N^\rho}
=
O(e^{-c'N})
\]
for some $c'>0$, with overwhelming probability.
} Consequently, $Z_N^{L(1)}$ and $Z_N^{L(3)}$ are negligible with respect to $Z_N^{L(2)}$. Hence
\begin{align}
    Z_N^L(\beta)
    &=
    \frac{\widetilde K_N N}{2}
    \left(\frac{\pi}{\beta}\right)^{N/2-1}
    \exp\left\{
        -\frac{N}{2}
        \int_{C_-}^{C_+}\log(C_+-x)\,d\nu(x)
        -N\beta_c(\lambda_1-C_+) + O_\prec(1)
    \right\}
    \nonumber \\
    &\quad \times
    e^{\beta N V(\lambda_1)}
    \sqrt{\frac{2\pi}{\beta N f''(\tau_{\lambda_1})}}
    \label{eq:LN-low-temp-final}
\end{align}
In particular,
\begin{equation} \label{eq:log-LN-low-temp}
    \frac{1}{N}\log Z_N^L(\beta)
    =
    \frac12\log\left(\frac{\pi}{\beta}\right)
    -
    \frac12\int_{C_-}^{C_+}\log(C_+-x)\,d\nu(x)
    -
    \beta_c(\lambda_1-C_+)
    +
    \beta V(\lambda_1)
    +
    O_\prec(N^{-1}).
\end{equation}

\subsection{Estimate of $Z_N^H(\beta)$ for
$\beta \in \left(\beta^*,\infty\right)$}

We now show that the bulk contribution $Z_N^H(\beta)$ is exponentially smaller than the edge contribution $Z_N^L(\beta)$ in the low-temperature regime. Choose $\gamma = \gamma_H$ where
\[
\gamma_H:=\lambda_1+N^{-1+4\epsilon}.
\]
We estimate the contour integral in the definition of $Z_N^H(\beta)$ on the vertical line $\rm{Re} z=\gamma_H$. Recall that
\[
R_N(z) = \int_0^{\beta_c/\beta} e^{N\beta(t z-f(t))}\,dt.
\]
For $z=\gamma_H+\mathrm{i}u$, we have
\[
\left| \exp\left\{ -\frac12\sum_{i=1}^N \log(z-\lambda_i) \right\} R_N(z) \right| \le \exp\left\{ -\frac12\sum_{i=1}^N\log(\gamma_H-\lambda_i) -\frac14\sum_{i=1}^N \log\left(1+\frac{u^2}{(\gamma_H-\lambda_i)^2}\right) \right\} R_N(\gamma_H).
\]
Since $\gamma_H=C_+ + O_\prec(N^{-2/3})$, with overwhelming probability, for all sufficiently large $N$,
\[
f'(s)<\gamma_H, \qquad 0\le s\le \frac{\beta_c}{\beta}.
\]
Thus the function $s\mapsto s\gamma_H-f(s)$ is increasing on $[0,\beta_c/\beta]$. Therefore
\[
\begin{aligned}
R_N(\gamma_H) &= \int_0^{\beta_c/\beta} e^{N\beta(t\gamma_H-f(t))}\,dt  \\ &\le \frac{\beta_c}{\beta} \exp\left\{ N\beta\left( \frac{\beta_c}{\beta}\gamma_H - f\left(\frac{\beta_c}{\beta}\right) \right) \right\}.
\end{aligned}
\]
Combining this with the standard bound on the $u$-integral, we obtain
\[
\begin{aligned}
|Z_N^H(\beta)| &\le C N \left(\frac{\pi}{\beta}\right)^{\frac N2} \exp\left\{ N\beta\left( \frac{\beta_c}{\beta}\gamma_H - f\left(\frac{\beta_c}{\beta}\right) \right) - \frac12\sum_{i=1}^N\log(\gamma_H-\lambda_i) \right\}
\end{aligned}
\]
with overwhelming probability.

We next use Corollary~\ref{cor:edge-approximation} to obtain
\[
\begin{aligned}
\frac1N\sum_{i=1}^N\log(\gamma_H-\lambda_i)
&= \int_{C_-}^{C_+}\log(C_+-x)\,d\nu(x) + 2\beta_c(\gamma_H-C_+) + O_\prec(N^{-1}).
\end{aligned}
\]
Therefore
\[
\begin{aligned}
\frac1N\log |Z_N^H(\beta)| &\le \frac12\log\left(\frac{\pi}{\beta}\right) - \frac12\int_{C_-}^{C_+}\log(C_+-x)\,d\nu(x)  \\ &\quad - \beta_c(\gamma_H-C_+) + \beta\left( \frac{\beta_c}{\beta}\gamma_H - f\left(\frac{\beta_c}{\beta}\right) \right) + O_\prec(N^{-1}).
\end{aligned}
\]
Since the terms containing $\gamma_H$ cancel, this bound can be rewritten as
\[
\begin{aligned}
\frac1N\log |Z_N^H(\beta)| &\le \frac12\log\left(\frac{\pi}{\beta}\right) - \frac12\int_{C_-}^{C_+}\log(C_+-x)\,d\nu(x) - \beta_c(\lambda_1-C_+) \\ &\quad + \beta\left( \frac{\beta_c}{\beta}\lambda_1 - f\left(\frac{\beta_c}{\beta}\right) \right) + O_\prec(N^{-1}).
\end{aligned}
\]
We will compare this with
\[
\begin{aligned}
\frac1N\log Z_N^L(\beta) &= \frac12\log\left(\frac{\pi}{\beta}\right) - \frac12\int_{C_-}^{C_+}\log(C_+-x)\,d\nu(x) - \beta_c(\lambda_1-C_+)  \\ &\quad + \beta V(\lambda_1) + O_\prec(N^{-1}).
\end{aligned}
\]
Comparing the last terms, we have $ V(\lambda_1) > \frac{\beta_c}{\beta}\lambda_1 - f\left(\frac{\beta_c}{\beta}\right) $ with overwhelming probability. This is because $\tau_{\lambda_1}$ remains separated from $\beta_c/\beta$ with overwhelming probability. Since $f$ is strongly convex, there exists a deterministic constant $c>0$ such that
\[
V(\lambda_1) - \left( \frac{\beta_c}{\beta}\lambda_1 - f\left(\frac{\beta_c}{\beta}\right) \right) \ge c
\]
with overwhelming probability, after decreasing $c>0$ if necessary.

Consequently,
\[
\frac1N\log |Z_N^H(\beta)| \le \frac1N\log Z_N^L(\beta)-c
\]
with overwhelming probability. Therefore
\[
|Z_N^H(\beta)|\le e^{-cN} Z_N^L(\beta)
\]
with overwhelming probability.

\section{Proof of the Main Results} \label{sec:proof}

In this section, we collect the estimates obtained in the previous sections and complete the proof of Theorem~\ref{thm:general-theorem}. The GOE case stated in Theorem~\ref{thm:specific-theorem} then follows by specializing to the semicircle law.

\begin{proof}
We first consider the high-temperature regime $0 < \beta < \beta^*$. By the estimates in Section~\ref{sec:high}, the bulk contribution $Z_N^H(\beta)$ dominates the edge contribution $Z_N^L(\beta)$.  More precisely,
\[
Z_N^L(\beta)=o(Z_N^H(\beta))
\]
with overwhelming probability, and hence
\[
\frac1N\log Z_N(\beta) = \frac1N\log Z_N^H(\beta)+o(N^{-1}).
\]
Lemma~\ref{lem:high-temp-approximation} gives
\[
\begin{aligned}
\frac1N\log Z_N^H(\beta) &= \frac12 \left[ J(\widehat\gamma_J) +\log\left(\frac{\pi}{\beta}\right) \right]  \\ &\quad +\frac{1}{2N} \log\left( \frac{2\beta}{f''(\tau_{\widehat{\gamma}_J})J''(\widehat{\gamma}_J)} \right) +o(N^{-1}),
\end{aligned}
\]
with overwhelming probability. Recall the definition $ J(\widehat\gamma_J) = 2\beta V(\widehat\gamma_J) - \frac1N\sum_{i=1}^N \log(\widehat\gamma_J-\lambda_i).$ We find that
\[
\begin{aligned}
F_N(\beta)
&= \beta V(\widehat\gamma_J) +\frac12\log\left(\frac{\pi}{\beta}\right) -\frac12\int_{C_-}^{C_+} \log(\widehat\gamma_J-x)\,d\nu(x) \\ &\quad -\frac{1}{2N} \mathcal{N}_{\varphi} -\frac{1}{2N} \log\left( 1 + \frac{f''(\tau_{\widehat\gamma_J})}{2\beta} \int^{C_+}_{C_-}\frac{d\nu(x)}{(\widehat\gamma_J - x)^2}  \right) \\ &\quad -\frac{1}{2N}\log \left( 1 - \frac{1}{N}\frac{1}{\frac{2\beta}{f''(\tau_{\widehat\gamma_J})} +  \int^{C_+}_{C_-}\frac{d\nu(x)}{(\widehat\gamma_J - x)^2}} \mathcal N_{\varphi''} \right) + o(N^{-1}).
\end{aligned}
\]
Here let $\varphi:\mathbb R\to\mathbb R$ be a compactly supported function which agrees with $x\mapsto \log(\widehat\gamma_J-x)$ on an open neighborhood of $[C_-,C_+]$. Since $\widehat\gamma_J>C_+$, such a function can be chosen to be analytic in a neighborhood of $[C_-,C_+]$. By Condition~\ref{cond:linear-statistics}, $\mathcal{N}_{\varphi}$ converges to the stated Gaussian law, while $\mathcal{N}_{\varphi''}=O_{\mathbb P}(1)$. Hence $\log(1 + \frac{C}{N}\mathcal{N}_{\varphi''})$ converges in probability to $0$, and Slutsky's theorem gives the desired result,
\[
\begin{aligned}
NF_N - N\left( \beta V(\widehat\gamma_J) +\frac12\log\left(\frac{\pi}{\beta}\right) - \frac12\int_{C_-}^{C_+} \log(\widehat\gamma_J-x)\,d\nu(x) \right) +\frac{1}{2} \log\left( 1 + \frac{f''(\tau_{\widehat\gamma_J})}{2\beta} \int^{C_+}_{C_-}\frac{d\nu(x)}{(\widehat\gamma_J - x)^2}  \right)\\ \Rightarrow \mathcal{N}\left(-\frac{1}{2}\rm{Mean}(\varphi),\frac{1}{4}\rm{Var}(\varphi) \right).
\end{aligned}
\]
Equivalently,
\[
N(F_N(\beta)-F(\beta)) \Rightarrow \mathcal N\left( -\frac12\rm{Mean}(\varphi) -\frac12\log\left( 1 + \frac{f''(\tau_{\widehat\gamma_J})}{2\beta} \int^{C_+}_{C_-}\frac{d\nu(x)}{(\widehat\gamma_J - x)^2}  \right), \frac14\rm{Var}(\varphi) \right).
\]
We next consider the low-temperature regime $\beta>\beta^*$. By Section~\ref{sec:low}, the edge contribution $Z_N^L(\beta)$ dominates the bulk contribution $Z_N^H(\beta)$.  That is, $Z_N^H(\beta)=o(Z_N^L(\beta))$ with overwhelming probability, and hence
\[
\frac1N\log Z_N(\beta) = \frac1N\log Z_N^L(\beta)+o(N^{-1}).
\]
Using the estimate \eqref{eq:log-LN-low-temp}, we obtain
\[
\begin{aligned}
F_N(\beta) &= \frac12\log\left(\frac{\pi}{\beta}\right) -\frac12\int_{C_-}^{C_+}\log(C_+-x)\,d\nu(x) -\beta_c(\lambda_1-C_+) +\beta V(\lambda_1) +O_\prec(N^{-1}).
\end{aligned}
\]
Since $V'(C_+)=\tau_{C_+}$ and $\lambda_1-C_+=O_\prec(N^{-2/3})$, Taylor expansion gives
\[
V(\lambda_1) = V(C_+) + \tau_{C_+}(\lambda_1-C_+) + O_\prec(N^{-4/3}).
\]
Therefore
\[
F_N(\beta) - \beta V(C_+) -\frac12\log\left(\frac{\pi}{\beta}\right) +\frac12\int_{C_-}^{C_+} \log(C_+-x)\,d\nu(x) = \bigl(\beta\tau_{C_+}-\beta_c\bigr) (\lambda_1-C_+) + O_\prec(N^{-1}).
\]
The Tracy--Widom limit now follows from Condition~\ref{cond:tracy-widom}.

Suppose now that $\tau_{C_+}>0$, so that $\beta^*<\infty$. It remains to consider the critical case $\beta=\beta^*$. Define
\[
F(\beta^*)
:=
\beta^*V(C_+)
+\frac12\log\left(\frac{\pi}{\beta^*}\right)
-\frac12\int_{C_-}^{C_+}\log(C_+-x)\,d\nu(x).
\]
Since \(\widehat\gamma(\beta)\downarrow C_+\) as \(\beta\uparrow\beta^*\), the high- and low-temperature expressions for \(F\) converge to the same value \(F(\beta^*)\). Hence \(F\) is continuous at \(\beta^*\).

Note that when $\beta^*<\infty$, the high- and low-temperature expressions for $F$ extend continuously to the same value at $\beta^*$. The convergence in probability at $\beta=\beta^*$ then follows from the convexity of $F_N$ and the continuity of $F$.
\end{proof}

{
\begin{proof}[Proof of Theorem~\ref{thm:specific-theorem}]
For the GOE, the limiting spectral measure is the semicircle law
\[
d\nu(x)=\frac{1}{2\pi}\sqrt{4-x^2}\,\mathbf 1_{[-2,2]}(x)\,dx.
\]
Hence
\[
C_+=2,\qquad s_\nu=\frac1\pi,\qquad \beta_c=\frac12.
\]
For $f(t)=\frac{\kappa}{2}t^2$,
\[
\tau_z=\frac{z}{\kappa},
\qquad
V(z)=\frac{z^2}{2\kappa} \quad \text{for } z>0,
\]
and therefore
\[
\beta^*=\frac{\beta_c}{\tau_{C_+}}=\frac{\kappa}{4}.
\]

Suppose first that $0<\beta<\kappa/4$, and set
\[
a:=\frac{2\beta}{\kappa}\in(0,1/2).
\]
Using the Stieltjes transform of the semicircle law, the saddle-point
equation gives
\[
\widehat\gamma
=
\frac{1}{\sqrt{a(1-a)}}.
\]
For
$\varphi(x)=\log(\widehat\gamma-x)$ on a neighborhood of $[-2,2]$,
the GOE linear-statistics formulas in
\cite[Section~3]{BaikLee2016} give
\[
\operatorname{Mean}(\varphi)
=
\frac12\log\left(\frac{1-2a}{1-a}\right),
\qquad
\operatorname{Var}(\varphi)
=
-2\log\left(\frac{1-2a}{1-a}\right).
\]
Moreover,
\[
\int_{-2}^{2}\frac{d\nu(x)}{(\widehat\gamma-x)^2}
=
\frac{a}{1-2a},
\]
so that
\[
1+\frac{\kappa}{2\beta}
\int_{-2}^{2}\frac{d\nu(x)}{(\widehat\gamma-x)^2}
=
\frac{2(1-a)}{1-2a}.
\]
Substituting these identities into Theorem~\ref{thm:general-theorem}(ii) yields
\[
\ell_\kappa(\beta)
=
\frac14\log\left(\frac{1-2a}{1-a}\right)
-\frac12\log2
=
\frac14\log\left(
\frac{1-4\beta/\kappa}{1-2\beta/\kappa}
\right)
-\frac12\log2,
\]
and
\[
\sigma_\kappa^2(\beta)
=
-\frac12\log\left(
\frac{1-4\beta/\kappa}{1-2\beta/\kappa}
\right).
\]

If $\beta>\kappa/4$, then
\[
(s_\nu\pi)^{2/3}=1,
\qquad
\beta\tau_{C_+}-\beta_c
=
\frac{2\beta}{\kappa}-\frac12.
\]
Thus Theorem~\ref{thm:general-theorem}(iii) gives precisely the low-temperature statement. The formula for the limiting free energy follows by substituting the semicircle law and the quadratic confinement into Theorem~\ref{thm:general-theorem}(i), which gives \eqref{eq:GOE-Flimit}.
\end{proof}
}

We now discuss the regularity of the limiting free energy. To study the regularity of the limiting free energy $F(\beta)$, additional regularity of the confinement function $f$ is needed. In the high-temperature regime, the critical point $\widehat{\gamma}=\widehat{\gamma}_J(\beta)$ varies with $\beta$ and is determined implicitly through
\[
\beta \tau_{\widehat{\gamma}} = \frac12\int_{C_-}^{C_+} \frac{d\nu(x)}{\widehat{\gamma}-x}, \qquad f'(\tau_z)=z.
\]
Thus, differentiating the high-temperature expression for $F(\beta)$ requires differentiability of the map $z\mapsto\tau_z$, and hence sufficient regularity of $f$ on the entire range of radial saddle points. As $\beta$ varies over the high-temperature regime, $\widehat{\gamma}_J(\beta)$ takes values in $(C_+,\infty)$, and consequently $\tau_{\widehat{\gamma}_J(\beta)}$ takes values in $(\tau_{C_+},\infty)$. The regularity at $\tau_{C_+}$ is also needed when matching the derivatives of the high- and low-temperature branches at the critical point. For this reason, in the following theorem we assume that $f$ is analytic on an open interval containing $[\tau_{C_+},\infty)$.

\begin{theorem}[Regularity of the limiting free energy]
Let $\tau_{C_+} > 0$. Suppose $f$ is analytic on an open interval containing $[\tau_{C_+},\infty)$. Then the limiting free energy $F(\beta)$ appearing in Theorem~\ref{thm:general-theorem} is $C^2$ but $\partial_\beta^3 F(\beta)$ is discontinuous at $\beta = \beta^*:=\frac{\beta_c}{\tau_{C_+}}$.
\end{theorem}

\begin{proof}
For $0<\beta<\beta^*$, let $\widehat\gamma_J=\widehat\gamma_J(\beta)\geq C_+$ be the unique solution of the saddle-point equation
\[
\beta\tau_{\widehat\gamma_J} = \frac12 \int_{C_-}^{C_+}\frac{d\nu(x)}{\widehat\gamma_J-x}.
\]
Take
\[
H(\beta,z)
=
\beta\tau_z-\frac12\int_{C_-}^{C_+}\frac{d\nu(x)}{z-x}.
\]
{
For $z>C_+$,
\[
\partial_z H(\beta,z)
=
\beta\tau'_z
+
\frac12\int_{C_-}^{C_+}\frac{d\nu(x)}{(z-x)^2}
>0.
\]
}
Hence the implicit function theorem shows that $\beta\mapsto\widehat\gamma_J(\beta)$ is smooth on $(0,\beta^*)$. We first determine the behavior of $\widehat\gamma_J$ as $\beta\uparrow\beta^*$. 

Set $a:=\widehat\gamma_J(\beta)-C_+$. Note that by the square-root behavior of $\nu$ in Condition~\ref{cond:regularity-measure}, as $a\downarrow0$,
\[ 
\begin{aligned}
\int_{C_-}^{C_+}
\frac{d\nu(x)}{(\widehat\gamma_J-x)^2} &= \frac{\pi s_\nu}{2\sqrt a}+O(1), \\
\int_{C_-}^{C_+} \frac{d\nu(x)}{(\widehat\gamma_J-x)^3} &= \frac{\pi s_\nu}{8a^{3/2}}+O(a^{-1/2}).
\end{aligned}
\]
Differentiating the saddle-point equation with respect to $\beta$, we obtain
\[
\tau_{\widehat\gamma_J}
+
\left(
\beta\tau'_{\widehat\gamma_J}
+
\frac12
\int_{C_-}^{C_+}
\frac{d\nu(x)}
{(\widehat\gamma_J-x)^2}
\right)
\widehat\gamma_J'(\beta)
=0.
\]
Denote 
\[
D(\beta)
:=
\beta\tau'_{\widehat\gamma_J}
+
\frac12
\int_{C_-}^{C_+}
\frac{d\nu(x)}
{(\widehat\gamma_J-x)^2}.
\]
Since $\tau_{\widehat\gamma_J}=\tau_{C_+} + O(a)$ and $D(\beta) = \frac{\pi s_\nu}{4\sqrt a}+O(1)$, we have
\[
\widehat\gamma_J'(\beta)
=
-\frac{4\tau_{C_+}}{\pi s_\nu}\sqrt a+O(a),
\]
and hence $\widehat\gamma_J'(\beta) \to 0$ as $\beta\uparrow\beta^*$. Differentiating once more, we have
\[
D'(\beta)
=
\tau'_{\widehat\gamma_J}
+
\beta\tau''_{\widehat\gamma_J}\widehat\gamma_J'
-
\left(
\int_{C_-}^{C_+}
\frac{d\nu(x)}
{(\widehat\gamma_J-x)^3}
\right)
\widehat\gamma_J'.
\]
Using the preceding estimates, $ D'(\beta) = \frac{\tau_{C_+}}{2a} + O(a^{-1/2})$. Therefore,
$$
\widehat\gamma_J''
=
-\frac{\tau'_{\widehat\gamma_J}\widehat\gamma_J'}{D}
+
\frac{\tau_{\widehat\gamma_J}D'}{D^2}.
$$
The first term converges to zero, while $D^2 = \frac{\pi^2s_\nu^2}{16a}\bigl(1+o(1)\bigr)$. Thus $\widehat\gamma_J''(\beta) \to \frac{8\tau_{C_+}^2}{\pi^2s_\nu^2}$ as $\beta\uparrow\beta^*$.



We now consider the derivatives of the limiting free energy. For $0<\beta<\beta^*$,
\[
F(\beta) = \beta V(\widehat\gamma_J) +\frac12\log\left(\frac{\pi}{\beta}\right) -\frac12\int_{C_-}^{C_+} \log(\widehat\gamma_J-x)\,d\nu(x).
\]
Differentiating and using $V'(z)=\tau_z$ together with the saddle-point equation, the terms involving $\widehat\gamma'_J(\beta)$ cancel, and we obtain
\[
F'(\beta) = V(\widehat\gamma_J)-\frac1{2\beta}, \qquad F''(\beta) = \tau_{\widehat\gamma_J}\widehat\gamma_J'(\beta) +\frac1{2\beta^2}.
\]
Since $\widehat\gamma_J\to C_+$ and $\widehat\gamma_J'(\beta)\to 0$,
\[
\lim_{\beta\uparrow\beta^*}F'(\beta) = V(C_+)-\frac{1}{2\beta^*}, \qquad \lim_{\beta\uparrow\beta^*}F''(\beta) = \frac1{2(\beta^*)^2}.
\]

On the other hand, for $\beta>\beta^*$,
\[
F(\beta) = \beta V(C_+) +\frac12\log\left(\frac{\pi}{\beta}\right) -\frac12\int_{C_-}^{C_+} \log(C_+-x)\,d\nu(x),
\]
so that
\[
F'(\beta) = V(C_+) -\frac{1}{2\beta}, \qquad F''(\beta)=\frac1{2\beta^2}.
\]
Thus
\[
\lim_{\beta\downarrow\beta^*}F'(\beta) = V(C_+)-\frac{1}{2\beta^*}, \qquad \lim_{\beta\downarrow\beta^*}F''(\beta) = \frac1{2(\beta^*)^2}.
\]
This shows that $F$ is $C^2$.

It remains to examine the third derivative. In the high-temperature regime,
\[
F'''(\beta) = \tau'_{\widehat\gamma_J} \bigl(\widehat\gamma_J'(\beta)\bigr)^2 + \tau_{\widehat\gamma_J}\widehat\gamma_J''(\beta) -\frac1{\beta^3}.
\]
Therefore,
\[
\lim_{\beta\uparrow\beta^*}F'''(\beta) = \frac{8\tau_{C_+}^3}{\pi^2s_\nu^2} -\frac1{(\beta^*)^3}.
\]
In the low-temperature regime,
\[
F'''(\beta)=-\frac1{\beta^3},
\]
and hence
\[
\lim_{\beta\downarrow\beta^*}F'''(\beta) = -\frac1{(\beta^*)^3}.
\]
Since $\tau_{C_+}>0$ and $s_\nu>0$, the two one-sided limits differ by
\[
\frac{8\tau_{C_+}^3}{\pi^2s_\nu^2}>0.
\]
Thus $F$ is $C^2$ but not $C^3$ at $\beta=\beta^*$.
\end{proof}
\subsection{Hard spherical limit}

We explain how the present soft model approaches the usual spherical Sherrington--Kirkpatrick model as the radial confinement becomes strong. Recall that the limiting free energy of the standard SSK model is
\[
F_{\rm{SSK}}(\beta) =
\begin{cases}
    \beta\gamma_{\rm{SSK}} - \frac{1}{2}\log(2\beta e) - \frac{1}{2}\int_{C_-}^{C_+} \log(\gamma_{\rm{SSK}} - x)\,d\nu(x) \quad &\text{if } 0<\beta<\beta_c, \\
    \beta C_+ -\frac{1}{2}\log(2\beta e) - \frac{1}{2}\int_{C_-}^{C_+}\log(C_+ - x)\,d\nu(x)   &\text{if } \beta > \beta_c,
\end{cases}
\]
where $\gamma_{\rm{SSK}}$ satisfies $\beta =\frac12 \int_{C_-}^{C_+} \frac{d\nu(x)}{\gamma_{\rm SSK}-x}$ and $\beta_c$ is defined by $\beta_c := \frac12 \int_{C_-}^{C_+} \frac{d\nu(x)}{C_+ - x}$.

We also remark that the radial variable in the soft model is $t = \frac{\norm{v}^2}{N}$. In the usual SSK model, the spin configurations satisfy $\norm{\sigma}^2 = N$, and hence the hard spherical constraint corresponds precisely to $t=1$. Therefore, in order to recover the spherical model from the present soft model, it is natural to choose a family of radial potentials whose mass becomes increasingly concentrated around $t=1$.

We consider
\[
f_\kappa(t)=\kappa\phi(t),
\]
where $\phi:\mathbb R_+\to\mathbb R$ is strongly convex, satisfies $\phi(1)=\phi'(1)=0$ and admits the local holomorphic extensions required in Theorem~\ref{thm:general-theorem}, so that, for each fixed $\kappa>0$, $f_\kappa=\kappa\phi$ satisfies the assumptions of that theorem. We further assume that $\phi''$ is continuous and strictly positive in a neighborhood of $1$. Thus $t=1$ is the unique minimizer of $\phi$. As $\kappa\to\infty$, the contribution from values of $t$ away from $1$ is increasingly suppressed, so that the radial variable concentrates near $t=1$.

For $z>C_+$, let
\[
\tau_z^{(\kappa)} := \arg\max_{t\geq0}\{tz-\kappa\phi(t)\}, \qquad V_\kappa(z) := z\tau_z^{(\kappa)} -\kappa\phi(\tau_z^{(\kappa)}).
\]
Since the critical point equation is $\kappa\phi'(\tau_z^{(\kappa)})=z$, $\phi'(1)=0$ and $\phi''(1)>0$, we find that
\[
\tau_z^{(\kappa)} = 1+O(\kappa^{-1}),
\]
uniformly for $z$ in any compact subset of $[C_+,\infty)$. Consequently,
\[
V_\kappa(z)=z+O(\kappa^{-1})
\]
locally uniformly. In particular, as $\kappa \to \infty$,
\[
\beta_\kappa^* = \frac{\beta_c}{\tau_{C_+}^{(\kappa)}} \to \beta_c.
\]

Let $F_\kappa(\beta)$ denote the limiting free energy of the soft model with confinement $f_\kappa$, and let $F_{\rm SSK}(\beta)$ denote the limiting free energy of the usual spherical SK model with normalized uniform measure on the sphere.

Suppose first that $0<\beta<\beta_c$. For all sufficiently large $\kappa$, we have $\beta<\beta_\kappa^*$. Let $\widehat\gamma_\kappa>C_+$ be the high-temperature saddle point satisfying
\[
\beta\tau_{\widehat\gamma_\kappa}^{(\kappa)} = \frac12 \int_{C_-}^{C_+} \frac{d\nu(x)} {\widehat\gamma_\kappa-x}.
\]
Since $\tau_z^{(\kappa)}\to1$ locally uniformly, $\widehat\gamma_\kappa$ converges to the unique solution $\gamma_{\rm SSK}>C_+$. Using the high-temperature formula in Theorem~\ref{thm:general-theorem} and $V_\kappa(\widehat\gamma_\kappa) =\widehat\gamma_\kappa+o(1)$, we obtain
\[
F_\kappa(\beta) \to F_{\rm SSK}(\beta) +\frac12\log(2\pi e).
\]

The additional constant does not come from the disorder or from the radial saddle itself, but from the different normalization of the two models.  The soft model is integrated with respect to Lebesgue measure on $\mathbb R^N$, whereas the spherical SK model is integrated with respect to the normalized uniform measure on the sphere $\|\sigma\|^2=N$.  Indeed, the polar decomposition used in Corollary~\ref{cor:soft-integral-representation} contains the deterministic factor
\[
C_N^{\rm Leb} = \frac{\pi^{N/2}N^{N/2}}{\Gamma(N/2)}.
\]
By Stirling's formula,
\[
\frac1N\log C_N^{\rm Leb} = \frac12\log(2\pi e)+o(1).
\]
Thus the term $\frac12\log(2\pi e)$ is the volume normalization associated with passing from normalized spherical measure to Lebesgue measure in radial coordinates. It is independent of the disorder and remains even when the radial variable concentrates at $t=1$.

The same conclusion holds in the low-temperature regime. If $\beta>\beta_c$, then $\beta>\beta_\kappa^*$ for all sufficiently large $\kappa$, and Theorem~\ref{thm:general-theorem} gives
\[
F_\kappa(\beta) = \beta V_\kappa(C_+) +\frac12\log\left(\frac{\pi}{\beta}\right) -\frac12 \int_{C_-}^{C_+} \log(C_+-x)\,d\nu(x).
\]
Since $V_\kappa(C_+)\to C_+$, we again obtain
\[
F_\kappa(\beta) \to F_{\rm SSK}(\beta) +\frac12\log(2\pi e).
\]
Consequently, for every fixed $\beta\neq\beta_c$,
\[
\lim_{\kappa\to\infty}F_\kappa(\beta) = F_{\rm SSK}(\beta) +\frac12\log(2\pi e).
\]
We next compare the fluctuations. For this discussion, we restrict to the standard random-matrix ensembles for which the limiting mean and variance of the linear spectral statistics are given by explicit formulas, such as the Wigner, orthogonally invariant, and real sample covariance ensembles considered in \cite[Section 3]{BaikLee2016}. For these ensembles, the corresponding mean and variance functionals are continuous under locally uniform convergence of analytic test functions on a fixed complex neighborhood of the limiting spectral support. In the high-temperature regime, Theorem~\ref{thm:general-theorem} gives
\[
N\bigl(F_{N,\kappa}(\beta)-F_\kappa(\beta)\bigr) \Rightarrow \mathcal N\bigl(\ell_\kappa(\beta), \sigma_\kappa^2(\beta)\bigr).
\]
Here,
\[
\sigma_\kappa^2(\beta) = \frac14 \rm{Var} \bigl(\varphi_\kappa\bigr)
\]
and
\[
\ell_\kappa(\beta) = -\frac12 \rm{Mean} \bigl(\varphi_\kappa \bigr) -\frac12 \log\left( 1+ \frac{ f_\kappa''( \tau_{\widehat\gamma_\kappa}^{(\kappa)}) }{2\beta} \int_{C_-}^{C_+} \frac{d\nu(x)} {(\widehat\gamma_\kappa-x)^2} \right),
\]
where $\varphi_\kappa:\R \to \R$ is a compactly supported function which agrees with $x \mapsto \log(\widehat\gamma_\kappa - x)$ on an open neighborhood of $[C_-,C_+]$. Since $\widehat\gamma_\kappa\to\gamma_{\rm SSK}$, we find that
\[
\sigma_\kappa^2(\beta) \to \frac14 \rm{Var} \bigl(\varphi_{\rm{SSK}}\bigr),
\]
where $\varphi_{\rm{SSK}}:\R \to \R$ is a compactly supported function which agrees with $x \mapsto \log(\gamma_{\rm{SSK}} - x)$ on an open neighborhood of $[C_-,C_+]$. This limit is precisely the variance of the high-temperature fluctuation of the usual SSK model.

The behavior of the mean is slightly different.  The reason is that the soft model still contains an integration in the radial variable. Near the radial saddle point, Laplace's method produces the Gaussian factor
\[
\left( \frac{2\pi} {N\beta f_\kappa''(\tau_{\widehat\gamma_\kappa}^{(\kappa)})} \right)^{1/2}.
\]
Since
\[
f_\kappa''(t)=\kappa\phi''(t)
\]
and $\tau_{\widehat\gamma_\kappa}^{(\kappa)}\to1$, this factor contains a deterministic factor of order $\kappa^{-1/2}$.  After taking the logarithm, its contribution to the free energy is of order
\[
-\frac{1}{2N}\log\kappa.
\]
This term vanishes in the limiting free energy, but it survives on the $N^{-1}$ fluctuation scale.  Accordingly,
\[
\ell_\kappa(\beta) = -\frac12\log\kappa+O(1).
\]
The divergence of $\ell_\kappa$ therefore does not represent an additional random fluctuation.  It is a deterministic normalization effect caused by the shrinking width of the radial shell around $t=1$.  In the hard spherical model the radius is fixed exactly, so there is no radial integration and no corresponding Laplace prefactor. Thus the random Gaussian part converges to that of the spherical SK model, while the mean contains an additional deterministic radial contribution.

The same radial Laplace factor is also present in the low-temperature regime.  In fact, the Laplace approximation of the radial integral contains
\[
\left( \frac{2\pi} {N\beta f_\kappa''(\tau_{\lambda_1}^{(\kappa)})} \right)^{1/2},
\]
and hence again produces a correction of order
\[
\frac{\log\kappa}{N}
\]
in the free energy.  However, the leading low-temperature fluctuation is of order $N^{-2/3}$ and is generated by the largest eigenvalue. Therefore, for each fixed $\kappa$,
\[
N^{2/3}\frac{\log\kappa}{N} = N^{-1/3}\log\kappa \to0
\]
as $N\to\infty$.  The deterministic radial normalization is therefore invisible at the Tracy--Widom scale.

Indeed, Theorem~\ref{thm:general-theorem} gives
\[
\frac{(s_\nu\pi N)^{2/3}}{\beta\tau_{C_+}^{(\kappa)}-\beta_c} \bigl(F_{N,\kappa}(\beta)-F_\kappa(\beta)\bigr) \Rightarrow \mathrm{TW}_1.
\]
Since $\tau_{C_+}^{(\kappa)}\to1$, $\beta\tau_{C_+}^{(\kappa)}-\beta_c \to \beta-\beta_c$ and hence the low-temperature Tracy--Widom fluctuation converges directly to that of the usual spherical SK model.

We emphasize that the fluctuation result is proved for each fixed $\kappa$ with $N\to\infty$, and the limit $\kappa\to\infty$ is taken afterwards.  No uniformity in $\kappa$ is claimed here.

In summary, the hard-constraint limit recovers the critical temperature and the random fluctuation mechanism of the usual spherical SK model.  At the level of the limiting free energy, the only difference is the deterministic volume normalization $\frac12\log(2\pi e)$.  At the fluctuation level, the radial Laplace factor produces an additional deterministic term in the high-temperature mean, while the same term is negligible on the $N^{-2/3}$ Tracy--Widom scale in the low-temperature regime.

\subsection*{Acknowledgments}
This work was partially supported by National Research Foundation of Korea under grant number RS-2023-NR076695.

\newpage
\appendix

\section{Auxiliary proofs and technical estimates} \label{sec:estimate}
\subsection{Approximation estimates}\label{app:approximation-estimates}

In this appendix, we prove Corollaries~\ref{cor:bulk-approximation} and~\ref{cor:edge-approximation}. We only use the rigidity estimate in Condition~\ref{cond:rigidity-eigenvalues} and the square-root behavior of the limiting density near the upper edge.

\begin{proof}[Proof of Corollary~\ref{cor:bulk-approximation}]
Fix $\delta>0$. For $z\ge C_++\delta$, the functions
\[
x\mapsto \frac1{z-x}, \qquad x\mapsto \log(z-x)
\]
are smooth on a fixed neighborhood of $[C_-,C_+]$, uniformly in $z\ge C_++\delta$. Recall that by Condition \ref{cond:rigidity-eigenvalues}, for any $\epsilon>0$, with overwhelming probability,
\[
|\lambda_i-\xi_i| \le \widehat{i}^{-1/3}N^{-2/3+\epsilon}, \qquad 1\le i\le N.
\]
Hence, uniformly in $z\ge C_++\delta$,
\[
\begin{aligned}
\left| \frac1N\sum_{i=1}^N \left(\frac1{z-\lambda_i}-\frac1{z-\xi_i}\right) \right| \le \frac{C}{N}\sum_{i=1}^N |\lambda_i-\xi_i| &\le C N^{-1+\epsilon},
\end{aligned}
\]
and similarly
\[
\left| \frac1N\sum_{i=1}^N \bigl(\log(z-\lambda_i)-\log(z-\xi_i)\bigr) \right| \le C N^{-1+\epsilon}.
\]
Since $\epsilon>0$ is arbitrary, these errors are $O_\prec(N^{-1})$.

It remains to compare the sums over the classical locations with the corresponding integrals. Let $q_i$ be defined by
\begin{equation} \label{eq:hat-gamma}
    \int_{q_i}^{C_+} d\nu = \frac{i}{N},\qquad i=1,\dots,N,
\end{equation}
with $q_0 = C_+$. Then $I_i:=[q_i,q_{i-1}]$ satisfies $\nu(I_i)=1/N$, and $\xi_i\in I_i$. Hence, for any $C^1$ function $\phi$ on $[C_-,C_+]$,
\[
\left| \frac1N\sum_{i=1}^N \phi(\xi_i) - \int_{C_-}^{C_+}\phi(x)\,d\nu(x) \right| \le \sum_{i=1}^N\int_{I_i} |\phi(\xi_i)-\phi(x)|\,d\nu(x) \le \frac{C_+-C_-}{N}\|\phi'\|_\infty.
\]
Applying this estimate to $\phi(x)=(z-x)^{-1}$ and $\phi(x)=\log(z-x)$, whose derivatives are uniformly bounded for $z\ge C_++\delta$, we obtain
\[
\frac1N\sum_{i=1}^N \frac1{z-\xi_i} = \int_{C_-}^{C_+}\frac{d\nu(x)}{z-x} + O(N^{-1}),
\]
and
\[
\frac1N\sum_{i=1}^N \log(z-\xi_i) = \int_{C_-}^{C_+}\log(z-x)\,d\nu(x) + O(N^{-1}).
\]
\end{proof}

\begin{proof}[Proof of Corollary~\ref{cor:edge-approximation}]
{
We first prove the resolvent estimate. Fix the exponent $\epsilon>0$ appearing in the statement of the corollary, and choose an auxiliary exponent $\eta>0$ such that $4\eta<\epsilon$. Since $z-\lambda_1\geq N^{-1+\epsilon}$, we also have
\[
z-\lambda_1\geq N^{-1+4\eta}
\]
for all sufficiently large $N$. In the estimates below, $\eta$ is used only as this auxiliary small exponent. The contribution of the sum over $i = 1,...,N^{3\eta}$ is negligible because
\[
\frac1N\sum_{i=1}^{N^{3\eta}}\frac1{z-\lambda_i} \leq N^{-\eta} = o(1).
\]
Moreover, the sum over $i =N - N^{3\eta}+1 ,..., N$ is also negligible since
\[
\frac1N \sum_{i=N - N^{3\eta}+1}^N \frac{1}{z-\lambda_i} \le CN^{-1+3\eta}.
\]
Thus it suffices to consider the sum over $ N^{3\eta}+1 \le i \le N - N^{3\eta}$. In the formula $z - \lambda_i = (z - \lambda_1) + (\lambda_1 - C_+) + (C_+ - \xi_i) +(\xi_i - \lambda_i)$, $C_+ - \xi_i$ is the dominating term by Condition~\ref{cond:rigidity-eigenvalues}. Hence
\[
\frac1N \sum_{i = N^{3\eta}+1}^{N - N^{3\eta}} \frac{1}{z-\lambda_i} = \frac1N \sum_{i = N^{3\eta}+1}^{N - N^{3\eta}} \frac{1}{(C_+ - \xi_i)}(1+ o(1))
\]
with overwhelming probability. It remains to justify the convergence of the deterministic sum. Since $x \mapsto \frac{1}{C_+ - x}$ on $x < C_+$ is an increasing function, we obtain
\[
    \frac1N \frac{1}{C_+ - \xi_i} \leq \int_{\xi_i}^{\xi_{i-1}} \frac{1}{C_+ - x} \; d\nu(x) \leq \frac1N \frac{1}{ C_+ - \xi_{i-1}}.
\]
Summing over $i$, we find that
\[
\begin{aligned}
     \left|
    \frac1N\sum_{i = N^{3\eta}+1}^{N - N^{3\eta}}\frac{1}{C_+ - \xi_i}
    -
    \int^{\xi_{N^{3\eta}}}_{\xi_{N-N^{3\eta}}} \frac{d\nu(x)}{C_+-x}
    \right|
    &\leq
    \frac{1}{N}\sum_{i = N^{3\eta}+1}^{N - N^{3\eta}} \left( \frac{1}{C_+ - \xi_{i-1}} - \frac{1}{C_+ - \xi_i} \right)     \\
    &= \frac{1}{N} \left( \frac{1}{C_+ - \xi_{N^{3\eta}}} - \frac{1}{C_+ - \xi_{N-N^{3\eta}}} \right) \\
    &= O(N^{-1/3-2\eta}).
\end{aligned}
\]
Hence,
\begin{align}  \label{eq:resolvent-det}
    \left|
    \frac1N\sum_{i = N^{3\eta}+1}^{N - N^{3\eta}}\frac{1}{C_+ - \xi_i}
    -
    \int_{C_-}^{C_+} \frac{d\nu(x)}{C_+-x}
    \right|
    &\leq
    O(N^{-1/3 - 2\eta})
    +
    C \left(
    \int_{C_-}^{\xi_{N - N^{3\eta}}}\frac{d\nu(x)}{C_+-x}
    +
    \int^{C_+}_{\xi_{N^{3\eta}}}\frac{d\nu(x)}{C_+-x}
    \right) \nonumber \\
    & = O(N^{-1/3+\eta}).
\end{align}
In conclusion, we have
\[
\frac1N\sum_{i=1}^N\frac1{z-\lambda_i} = 2\beta_c + o(1).
\]

We now prove the logarithmic estimate. Using the same argument as above, we have
\[
\begin{aligned}
\left| \frac1N\sum_{i=1}^{N^{3\eta}} \log(z-\lambda_i)\right| &\le C\frac{N^{3\eta}\log N}{N}, \\ \left| \frac1N\sum_{i=N - N^{3\eta} +1}^{N} \log(z-\lambda_i)\right| &\le CN^{-1+3\eta}.
\end{aligned}
\]
For $N^{3\eta} < i < N - N^{3\eta}$ we may expand
\[
\log(z-\lambda_i) = \log(C_+-\xi_i) + \frac{z-C_+}{C_+-\xi_i} + R_i,
\]
where $R_i$ is the remainder term. Note that $R_i$ satisfies
\[
\begin{aligned}
R_i = \log\left( 1 + \frac{z - C_+}{C_+ - \xi_i} + \frac{\xi_i - \lambda_i}{C_+ - \xi_i} \right) - \frac{z-C_+}{C_+-\xi_i} = O\left(\left(  \frac{z - C_+}{C_+ - \xi_i} \right)^2\right) + O\left( \frac{\lambda_i - \xi_i}{C_+ - \xi_i}\right).
\end{aligned}
\]

Using Condition \ref{cond:rigidity-eigenvalues}, the first error term can be estimated by
\[
O\left(\left(  \frac{z - C_+}{C_+ - \xi_i} \right)^2\right) \le C\frac{N^{-4/3 + 2\eta}}{i^{4/3}N^{-4/3}},
\]
and the second error term satisfies
\[
\left|
\frac{\lambda_i-\xi_i}{C_+-\xi_i}
\right|
\le
\begin{cases}
C N^\eta i^{-1},
&\text{for }  N^{3\eta}<i\le N/2, \\
C N^{-2/3+\eta}(N+1-i)^{-1/3},
&\text{for } N/2<i<N-N^{3\eta},
\end{cases}
\]
with overwhelming probability. Hence summing over $i$ gives
\[
\frac1N \sum_{i=N^{3\eta}+1}^{N - N^{3\eta}} |R_i| = O_\prec(N^{-1}).
\]
Now the terms $\log(C_+ - \xi_i)$ and $\frac{z-C_+}{C_+ - \xi_i}$ should be estimated. The logarithmic term is given by
\[
\begin{aligned}
&\left| \frac1N \sum_{i=N^{3\eta}+1}^{N - N^{3\eta}} \log (C_+ - \xi_i) - \int_{C_-}^{C_+} \log (C_+ - x) d\nu(x) \right| \\
&\le \frac{1}{N}\pbb{\log(C_+ - \xi_{N-N^{3\eta}}) - \log(C_+ - \xi_{N^{3\eta}})} + {C\left(
\left|\int_{C_-}^{\xi_{N-N^{3\eta}}}\log(C_+-x)\,d\nu(x)\right|
+
\left|\int_{\xi_{N^{3\eta}}}^{C_+}\log(C_+-x)\,d\nu(x)\right|
\right)} \\ 
&\le O(N^{-1}\log N) + C \pB{N^{-1 + 3\eta} + \int_0^{N^{-2/3+2\eta}} \sqrt{u}|\log u| \; du} = O_\prec(N^{-1}).
\end{aligned}
\]
For the resolvent term, \eqref{eq:resolvent-det} gives
\[
\frac{z-C_+}{N} \sum_{i=N^{3\eta}+1}^{N-N^{3\eta}}\frac1{C_+-\xi_i} = 2\beta_c(z-C_+) + O_\prec(N^{-1}).
\]
Consequently,
\[
\frac1N\sum_{i=1}^N\log(z-\lambda_i) = \int_{C_-}^{C_+}\log(C_+-x)\,d\nu(x) + (z-C_+)\int_{C_-}^{C_+}\frac{d\nu(x)}{C_+-x} + O_\prec(N^{-1}).
\]
{
Here the auxiliary exponent $\eta>0$ may be chosen arbitrarily small subject to $4\eta<\epsilon$; hence the displayed deterministic $N^{-1+O(\eta)}$ errors are absorbed into the notation $O_\prec(N^{-1})$.
}
}
\end{proof}

\subsection{Integral representation}\label{app:integral-representation}
\subsubsection*{Proof of Lemma~\ref{lem:integral-representation}}
\begin{proof}
We consider the following integral:
\begin{equation}
    L(z) := \int_{\R^N} e^{\beta N \sum_{i}(\lambda_i - z)y_i^2} d\mathbf{y}, \qquad \re{z} > \lambda_1.
\end{equation}
We evaluate this integral in two different ways: first by computing the Gaussian integral coordinatewise, and then by using polar coordinates.

First, evaluating the Gaussian integral coordinatewise gives
\[
L(z) = \left( \frac{\pi}{\beta N} \right)^{N/2} \prod_{i=1}^N \frac{1}{\sqrt{z-\lambda_i}}.
\]
On the other hand, applying the polar decomposition
\[
\mathbf{y} = \sqrt{\frac{t}{\beta N}}\mathbf{x},\quad t = \beta N \norm{\mathbf{y}}^2, \quad \mathbf{x}\in S^{N-1} := \{ \mathbf{x} \in \R^N : \norm{\mathbf{x}}^2 = 1\}
\]
gives
\[
L(z) = \frac{|S^{N-1}|}{2(\beta N)^{N/2}}\int_0^\infty e^{-zt}t^{N/2-1}I(t)\,dt \qquad I(t):= \int_{S^{N-1}} e^{t\sum_i \lambda_i x_i^2} \, d\omega_N(\mathbf{x}).
\]
Combining the two expressions for $L(z)$, we obtain
\[
\Gamma(N/2) \prod_{i=1}^N \frac{1}{\sqrt{z-\lambda_i}} = \int_0^\infty e^{-zt}t^{N/2-1}I(t)\,dt.
\]
Note that the right-hand side is the Laplace transform of $t^{(N/2)-1}I(t)$. Thus applying inverse Laplace transform gives
\[
t^{N/2-1}I(t) = \frac{\Gamma(N/2)}{2\pi \mathrm{i}} \int_{\gamma - \mathrm{i}\infty}^{\gamma+ \mathrm{i}\infty} e^{zt} \prod_{i=1}^N \frac{1}{\sqrt{z-\lambda_i}} \, dz.
\]
Since the real symmetric $N \times N$ matrix $M$ can be diagonalized by $M = O^TDO$ where $O$ is an orthogonal matrix and $D = \rm{diag}(\lambda_1,...,\lambda_N)$, the rotational invariance of the integral on $S^{N-1}$ gives
\[
\int_{S_{N-1}} e^{\beta \langle \sigma, M\sigma \rangle} d\omega_N(\sigma) = \int_{S^{N-1}} e^{\beta N \sum_{i} \lambda_i x_i^2} \,d\omega_N(\mathbf{x}) = I(\beta N).
\]
Hence, taking $t = \beta N$ gives the desired result,
\[
I(\beta N) = C_N \int_{\gamma - \mathrm{i}\infty}^{\gamma+ \mathrm{i}\infty} e^{\frac{N}{2}G_0(z)} \, dz.
\]
\end{proof}

\subsection{High-temperature saddle and contour estimates}
\label{app:high-temp-technical}

This subsection contains the technical estimates used in Section~\ref{sec:high}. They are included for completeness; the arguments are standard away from the spectral edge, except that the radial factor $I_N$ is treated separately in the main text.

\subsubsection*{Proof of Lemma~\ref{lem:J-derivative-estimates}}
\begin{proof}
For real $x\ge C_+ +\delta$, we have $ J'(x) = 2\beta \tau_x - \frac1N\sum_{i=1}^N\frac{1}{x-\lambda_i}, $ whereas $ \widehat J'(x) = 2\beta \tau_x - \int_{C_-}^{C_+}\frac{d\nu(k)}{x-k}. $ Hence
\[
J'(x)-\widehat J'(x) = -\left( \frac1N\sum_{i=1}^N\frac{1}{x-\lambda_i} - \int_{C_-}^{C_+}\frac{d\nu(k)}{x-k} \right).
\]
By the resolvent approximation in Corollary~\ref{cor:bulk-approximation}, the right-hand side is $\prec N^{-1}$ uniformly for $x\ge C_+ +\delta$ with overwhelming probability. This proves (i).

We next prove (ii). Since $K\subset U_\eta$ is compact and $V$ is holomorphic on $U_\eta$, all derivatives $V^{(\ell)}$ are bounded on $K$. Moreover, by choosing $\eta>0$ sufficiently small, $U_\eta$ stays a positive distance away from the spectral support $[C_-,C_+]$. Thus, with overwhelming probability,
\[
\inf_{z\in K}\min_{1\le i\le N}|z-\lambda_i|\ge c_K
\]
for some constant $c_K>0$ independent of $N$. For $\ell\ge 2$,
\[
J^{(\ell)}(z) = 2\beta V^{(\ell)}(z) - \frac{(-1)^{\ell-1}(\ell-1)!}{N} \sum_{i=1}^N\frac{1}{(z-\lambda_i)^\ell}.
\]
Therefore,
\[
\sup_{z\in K}|J^{(\ell)}(z)| \le C_K + \frac{C_\ell}{N}\sum_{i=1}^N c_K^{-\ell} = O(1).
\]
This proves (ii).
\end{proof}

\subsubsection*{Proof of Lemma~\ref{lem:gamma-approximation}}
\begin{proof}
We follow the proof of Corollary 5.2 in \cite{BaikLee2016}. Since $0 < \beta < \beta^*$, Lemma~\ref{lem:critical-point-Jhat} implies that $\widehat\gamma_J>C_+$. Choose $\delta>0$ such that $\widehat\gamma_J-\delta>C_+$. Then Lemma~\ref{lem:J-derivative-estimates}(i) applies in a neighborhood of $\widehat\gamma_J$. Fix small $\epsilon>0$. By Lemma~\ref{lem:J-derivative-estimates}(i), uniformly for $z\in[\widehat\gamma_J-N^{-1+2\epsilon},\widehat\gamma_J+N^{-1+2\epsilon}]$, we have
\[
J'(z)=\widehat J'(z)+O(N^{-1+\epsilon})
\]
with overwhelming probability. Since $\widehat J'(\widehat\gamma_J)=0$, Taylor expansion gives
\[
\begin{aligned}
\widehat J'(\widehat\gamma_J\pm N^{-1+2\epsilon}) &= \widehat J''(\widehat\gamma_J)(\pm N^{-1+2\epsilon}) + O(N^{-2+4\epsilon}).
\end{aligned}
\]
Moreover, $ \widehat J''(\widehat\gamma_J) = 2\beta V''(\widehat\gamma_J) + \int_{C_-}^{C_+} \frac{d\nu(k)}{(\widehat\gamma_J-k)^2} >0. $ Thus
\[
\widehat J'(\widehat\gamma_J+N^{-1+2\epsilon}) = \widehat J''(\widehat\gamma_J)N^{-1+2\epsilon}+O(N^{-2+4\epsilon})>0 \quad \text{and} \quad \widehat J'(\widehat\gamma_J-N^{-1+2\epsilon}) = -\widehat J''(\widehat\gamma_J)N^{-1+2\epsilon}+O(N^{-2+4\epsilon})<0
\]
for all sufficiently large $N$. Since $N^{-1+2\epsilon}$ dominates the error term $N^{-1+\epsilon}$, we obtain $J'(\widehat\gamma_J+N^{-1+2\epsilon})>0$ and $J'(\widehat\gamma_J-N^{-1+2\epsilon})<0$ with overwhelming probability.

As in the proof of Lemma~\ref{lem:critical-point-Jhat}, the function $J'$ is strictly increasing on the whole interval $(\lambda_1,\infty)$ because $x\mapsto\tau_x$ is nondecreasing while $x\mapsto N^{-1}\sum_i(x-\lambda_i)^{-1}$ is strictly decreasing. This argument does not require $V''$ to be defined globally. Locally near $\widehat\gamma_J$, where the holomorphic extension is available, we still have
\[
J''(z)
=
2\beta V''(z)+\frac1N\sum_{i=1}^N\frac{1}{(z-\lambda_i)^2}>0,
\]
for real $z$.
Therefore the unique critical point $\gamma_J$ of $J$ satisfies
\[
\widehat\gamma_J-N^{-1+2\epsilon}
<
\gamma_J
<
\widehat\gamma_J+N^{-1+2\epsilon}
\]
with overwhelming probability. Since $\epsilon>0$ is arbitrary, this proves $ |\gamma_J-\widehat\gamma_J|\prec N^{-1}. $

The second claim follows directly from Lemma~\ref{lem:critical-point-Jhat}, $\widehat\gamma_J>C_+$, and by the convergence of the largest eigenvalue to the spectral edge.
\end{proof}

\subsubsection*{Proof of Lemma~\ref{lem:J-saddle-comparison}}
\begin{proof}
By Lemma~\ref{lem:gamma-approximation}, $|\gamma_J-\widehat\gamma_J|\prec N^{-1}$. Since $\widehat\gamma_J>C_+$, we may choose a fixed compact set $K\subset U_\eta$ such that the interval between $\gamma_J$ and $\widehat\gamma_J$ is contained in $K$ with overwhelming probability. Lemma~\ref{lem:J-derivative-estimates}(ii) then gives $J''=O(1)$ and $J'''=O(1)$ uniformly on $K$.

Now, using $J'(\gamma_J)=0$, Taylor expansion yields
\[
J(\widehat\gamma_J)-J(\gamma_J) = \frac12 J''(\xi)(\widehat\gamma_J-\gamma_J)^2
\]
for some $\xi$ between $\gamma_J$ and $\widehat\gamma_J$. Hence, using Lemma~\ref{lem:gamma-approximation} and the uniform bound on $J''$, we obtain
\[
J(\gamma_J)-J(\widehat\gamma_J)=O_\prec(N^{-2})
\]
with overwhelming probability. Similarly,
\[
J''(\gamma_J)-J''(\widehat\gamma_J) = J'''(\zeta)(\gamma_J-\widehat\gamma_J) = O_\prec(N^{-1}),
\]
where $\zeta$ lies between $\gamma_J$ and $\widehat\gamma_J$. This proves the lemma.
\end{proof}

\subsubsection*{Proof of Lemma~\ref{lem:high-temp-approximation}}
\begin{proof}
We prove the estimate on the overwhelming probability event where Lemmas \ref{lem:gamma-approximation}, \ref{lem:J-saddle-comparison}, and \ref{lem:I-estimate} hold.

We first split the last integral into the local part $|y|\le N^\epsilon$ and the tail part $|y|>N^\epsilon$. We can write the contour integral in the scaled variable $z=\gamma_J+\frac{\mathrm{i}y}{\sqrt N}$. For the local part, we assume $z \in U_\eta$ which holds for large $N$. Then
\begin{align}
    \int_{\gamma_J-\mathrm{i}\frac{N^\epsilon}{\sqrt{N}}}^{\gamma_J+\mathrm{i}\frac{N^\epsilon}{\sqrt{N}}} e^{\frac{N}{2}J(z)}I_N(z)\,dz
    &=
    \frac{\mathrm{i}}{\sqrt N}
    \int_{-N^\epsilon}^{N^\epsilon}
    e^{\frac{N}{2}J(\gamma_J+\frac{\mathrm{i}y}{\sqrt N})}
    I_N\left(\gamma_J+\frac{\mathrm{i}y}{\sqrt N}\right)\,dy  \notag \\
    &=
    \frac{\mathrm{i} e^{\frac{N}{2}J(\gamma_J)}}{\sqrt N}
    \int_{-N^\epsilon}^{N^\epsilon}
    e^{\frac{N}{2}\left(
     J(\gamma_J+\frac{\mathrm{i}y}{\sqrt N})-J(\gamma_J)
    \right)}
     I_N\left(\gamma_J+\frac{\mathrm{i}y}{\sqrt N}\right)\,dy .
    \label{eq:contour-scaled}
\end{align}

Since $J'(\gamma_J)=0$, Taylor expansion gives
\begin{align}
    J\left(\gamma_J+\frac{\mathrm{i}y}{\sqrt N}\right)-J(\gamma_J)
    &=
    \frac{J''(\gamma_J)}{2}
    \left(\frac{\mathrm{i}y}{\sqrt N}\right)^2
    +
    \frac{J'''(\gamma_J)}{6}
    \left(\frac{\mathrm{i}y}{\sqrt N}\right)^3
    +
    O\left(\frac{|y|^4}{N^2}\right)  \notag \\
    &=
    -\frac{J''(\gamma_J)y^2}{2N}
    -
    \mathrm{i}\frac{J'''(\gamma_J)y^3}{6N^{3/2}}
    +
    O(N^{-2+4\epsilon}),
    \label{eq:J-local-expansion}
\end{align}
uniformly for $|y|\le N^\epsilon$. Hence
\[
e^{\frac{N}{2}\left( J(\gamma_J+\frac{\mathrm{i}y}{\sqrt N})-J(\gamma_J) \right)} = e^{-\frac{J''(\gamma_J)}{4}y^2} \left( 1-\mathrm{i}\frac{J'''(\gamma_J)}{12}\frac{y^3}{\sqrt N} +O(N^{-1+6\epsilon}) \right),
\]
uniformly for $|y|\le N^\epsilon$. Combining the result of Lemma \ref{lem:I-estimate} gives
\begin{align}
    &\int_{-N^\epsilon}^{N^\epsilon}
    e^{\frac{N}{2}\left(
    J(\gamma_J+\frac{\mathrm{i}y}{\sqrt N})-J(\gamma_J)
    \right)}
    I_N\left(\gamma_J+\frac{\mathrm{i}y}{\sqrt N}\right)\,dy  \notag\\
    &\quad =
    \sqrt{\frac{2\pi}{\beta N f''(\tau_{\gamma_J})}}
    \left(1+O(N^{-1/4+C\epsilon})\right)
    \int_{-N^\epsilon}^{N^\epsilon}
    e^{-\frac{J''(\gamma_J)}{4}y^2}
    \left(
        1-\mathrm{i}\frac{J'''(\gamma_J)}{12}\frac{y^3}{\sqrt N}
        +O(N^{-1+6\epsilon})
    \right)\,dy .
    \label{eq:local-integral-before-gaussian}
\end{align}
Note that the integral of the odd term vanishes:
\[
\int_{-N^\epsilon}^{N^\epsilon} y^3 e^{-\frac{J''(\gamma_J)}{4}y^2}\,dy=0.
\]
Since $J''(\gamma_J)$ is bounded below by a positive constant with overwhelming probability, the Gaussian tails outside $[-N^\epsilon,N^\epsilon]$ are exponentially small. Hence
\[
\int_{-N^\epsilon}^{N^\epsilon} e^{-\frac{J''(\gamma_J)}{4}y^2}\,dy = \int_{-\infty}^{\infty} e^{-\frac{J''(\gamma_J)}{4}y^2}\,dy +O(e^{-cN^{2\epsilon}}) = \frac{2\sqrt{\pi}}{\sqrt{J''(\gamma_J)}} +O(e^{-cN^{2\epsilon}}).
\]
Combining this with \eqref{eq:local-integral-before-gaussian}, we obtain
\begin{equation} \label{eq:local-integral-estimate}
    \int_{-N^\epsilon}^{N^\epsilon}
    e^{\frac{N}{2}\left(
    J(\gamma_J+\frac{\mathrm{i}y}{\sqrt N})-J(\gamma_J)
    \right)}
    I_N\left(\gamma_J+\frac{\mathrm{i}y}{\sqrt N}\right)\,dy
    =
    \frac{2\sqrt{2}\pi}
    {\sqrt{\beta N f''(\tau_{\gamma_J})J''(\gamma_J)}}
    \left(1+O(N^{-1/4+C\epsilon})\right)
\end{equation}

It remains to show that the contribution from $|y|>N^\epsilon$ is negligible. On this part we do not use the local representation $R_N(z)=e^{N\beta V(z)}I_N(z)$ since $V$ is only defined in the neighborhood $U_\eta$.  Instead, we use the original expression for $R_N$ as in \eqref{eq:R-N-def}. Put $z=\gamma_J+\mathrm{i}y/\sqrt N$.  Since $\gamma_J$ stays a fixed positive distance away from the spectrum, with overwhelming probability
\[
|z-\lambda_i|^2 =(\gamma_J-\lambda_i)^2+\frac{y^2}{N} \geq |\gamma_J-\lambda_i|^2\left(1+\frac{c y^2}{N}\right)
\]
for all $i$. Hence
\[
\prod_{i=1}^N |z-\lambda_i|^{-1/2} \leq \prod_{i=1}^N |\gamma_J-\lambda_i|^{-1/2} \left(1+\frac{c y^2}{N}\right)^{-N/4}.
\]
Moreover, by the original definition of $R_N$ and the estimate \eqref{eq:psi-property}, we have
\[
|R_N(\gamma_J+\frac{\mathrm{i}y}{\sqrt N})| \leq \int_0^{\beta_c/\beta} e^{N\beta(u\gamma_J-f(u))}\,du \leq C N^{-1/2} e^{N\beta V(\gamma_J)}
\]
with overwhelming probability. Therefore the tail contribution is bounded by
\[
C\frac{e^{N\beta V(\gamma_J)}}{N} \prod_{i=1}^N |\gamma_J-\lambda_i|^{-1/2} \int_{|y|>N^\epsilon} \left(1+\frac{c y^2}{N}\right)^{-N/4}\,dy .
\]
The last integral is bounded by $Ce^{-cN^{2\epsilon}}$, and hence the tail contribution is negligible compared with the stated error term. Combining the two estimates, we conclude that after removing the factor $dz=\rm{i}\,dy/\sqrt N$, the remaining $dy$-integral satisfies 
\[
\left| e^{-\frac{N}{2}J(\gamma_J)} \int_{|y|>N^\epsilon} \exp\left\{ -\frac12\sum_{i=1}^N\log(\gamma_J+\frac{\mathrm{i}y}{\sqrt N}-\lambda_i) \right\} R_N(\gamma_J+\frac{\mathrm{i}y}{\sqrt N}) \,dy \right| = o(N^{-1/2}).
\]

Combining this tail estimate with \eqref{eq:local-integral-estimate} and \eqref{eq:contour-scaled}, we find
\[
\int_{\gamma_J-\mathrm{i}\infty}^{\gamma_J+\mathrm{i}\infty} \exp\left\{ -\frac12\sum_{i=1}^N\log(z-\lambda_i) \right\} R_N(z) \, dz = \mathrm{i} e^{\frac{N}{2}J(\gamma_J)} \frac{2\sqrt{2}\pi} {N\sqrt{\beta f''(\tau_{\gamma_J})J''(\gamma_J)}} \left(1+O(N^{-1/4+C\epsilon})\right)
\]
Finally, by Lemmas \ref{lem:J-saddle-comparison} and \ref{lem:gamma-approximation} we have
\[
J(\gamma_J)=J(\widehat\gamma_J)+O_\prec(N^{-2}), \quad J''(\gamma_J) = J''(\widehat\gamma_J) + O_\prec(N^{-1}), \quad f''(\tau_{\gamma_J})= f''(\tau_{\widehat\gamma_J}) + O_\prec(N^{-1}),
\]
with overwhelming probability. Putting into the equation, the errors are absorbed into $O_\prec(N^{-1/4+C\epsilon})$, so we obtain the desired result,
\[
\int_{\gamma_J-\mathrm{i}\infty}^{\gamma_J+\mathrm{i}\infty} \exp\left\{ -\frac12\sum_{i=1}^N\log(z-\lambda_i) \right\} R_N(z) \, dz = \mathrm{i} e^{\frac{N}{2}J(\widehat\gamma_J)} \frac{2\sqrt{2}\pi} {N\sqrt{\beta f''(\tau_{\widehat\gamma_J})J''(\widehat\gamma_J)}} \left(1+O(N^{-1/4+C\epsilon})\right)
\]
\end{proof}

\subsection{Uniform low-temperature contour estimates}
\label{app:low-temp-technical}

We now collect the fixed-radius estimates used in Section~\ref{sec:low}.  For each fixed $t$, the phase is the same as in the low-temperature analysis of the spherical model after replacing the inverse temperature by $\beta t$.  The point here is to record the estimates uniformly for $t\in[\beta_c/\beta+\delta,T]$. Throughout this subsection, for fixed $t$ we write $\gamma_G=\gamma_G(t)$ for the unique critical point of $G(\cdot;t)$.

\subsubsection*{Proof of Lemma~\ref{lem:critical-point-G}}
\begin{proof}
We first note that
\[
G'(z;t) = 2\beta t-\frac{1}{N}\sum_{i=1}^N\frac{1}{z-\lambda_i} \quad \text{and} \quad G''(z;t) = \frac{1}{N}\sum_{i=1}^N\frac{1}{(z-\lambda_i)^2}>0
\]
for $z>\lambda_1$. Hence $G'(z;t)$ is strictly increasing on $(\lambda_1,\infty)$. Moreover,
\[
\lim_{z\downarrow \lambda_1}G'(z;t)=-\infty, \qquad \lim_{z\to\infty}G'(z;t)=2\beta t>0.
\]
Thus there exists a unique critical point $\gamma_G=\gamma_G(t)>\lambda_1$.

We prove the lower bound first. Set $ z_-=\lambda_1+\frac{1}{3\beta Nt}. $ Then
\[
G'(z_-;t) = 2\beta t-\frac{1}{N}\sum_{i=1}^N\frac{1}{z_- -\lambda_i} \leq 2\beta t-\frac{1}{N}\frac{1}{z_- -\lambda_1} = 2\beta t-3\beta t = -\beta t<0.
\]
Since $G'(\cdot;t)$ is strictly increasing and $G'(\gamma_G;t)=0$, we have $\gamma_G>z_-$. This proves $ \gamma_G-\lambda_1 \geq \frac{1}{3\beta Nt}. $

We next prove the upper bound. Set $ z_+=\lambda_1+N^{-1+\epsilon}. $ Since $0<\epsilon<1/12$, we have
\[
z_+\in \left( \lambda_1+N^{-1}, \lambda_1+O_\prec(N^{-2/3}) \right).
\]
Hence, by Corollary \ref{cor:edge-approximation},
\[
\frac{1}{N}\sum_{i=1}^N\frac{1}{z_+-\lambda_i} = \int_{C_-}^{C_+}\frac{d\nu(k)}{C_+-k} +o(1) = 2\beta_c+o(1)
\]
with overwhelming probability. Therefore,
\begin{align*}
    G'(z_+;t)
    &=
    2\beta t
    -
    \frac{1}{N}\sum_{i=1}^N\frac{1}{z_+-\lambda_i} \\
    &=
    2\beta t-2\beta_c+o(1) > 0
\end{align*}
with overwhelming probability. Since $G'(\cdot;t)$ is strictly increasing and $G'(\gamma_G;t)=0$, it follows that $\gamma_G<z_+$. Hence $\gamma_G-\lambda_1 \prec N^{-1}$. This completes the proof.
\end{proof}

\subsubsection*{Proof of Lemma~\ref{lem:G-approximation}}
\begin{proof}
Throughout the proof, we work on the overwhelming probability event on which Lemma \ref{lem:critical-point-G}, Corollary \ref{cor:bulk-approximation}, and Corollary \ref{cor:edge-approximation} hold uniformly in the relevant domain.

We first prove \eqref{eq:G-gamma-approx-low}. We can write
\begin{align*}
    G(\gamma_G;t)
    =
    2\beta t\lambda_1
    +
    2\beta t(\gamma_G-\lambda_1)
    -
    \frac{1}{N}\sum_{i=1}^N\log(\gamma_G-\lambda_i).
\end{align*}

By Lemma \ref{lem:critical-point-G}, $ 2\beta t(\gamma_G-\lambda_1)=O_\prec(N^{-1}) $ uniformly in $t$. Moreover applying Corollary \ref{cor:edge-approximation} gives
\[
\frac{1}{N}\sum_{i=1 }^N\log(\gamma_G-\lambda_i) = \int_{C_-}^{C_+}\log(C_+-x)\,d\nu(x) + 2\beta_c(\lambda_1-C_+) + O_\prec(N^{-1}),
\]
uniformly in $t\in[\beta_c/\beta+\delta,T]$. Combining the results, we obtain
\[
G(\gamma_G;t) = 2\beta t\lambda_1 - \int_{C_-}^{C_+}\log(C_+-x)\,d\nu(x) - 2\beta_c(\lambda_1-C_+) + O_\prec(N^{-1}),
\]
which proves \eqref{eq:G-gamma-approx-low}. We next prove \eqref{eq:G-derivative-bound-low}. For any fixed $0 < \epsilon < 1/12$,  $\gamma_G(t) - \lambda_1 \leq N^{-1+4\epsilon}$ holds with overwhelming probability. For $l\geq2$, direct differentiation gives
\[
\frac{(-1)^l}{(l-1)!}G^{(l)}(\gamma_G;t) = \frac{1}{N}\sum_{i=1}^N \frac{1}{(\gamma_G-\lambda_i)^l}.
\]
The lower bound follows from the $i=1$ term as follows.
\[
\frac{1}{N}\sum_{i=1}^N \frac{1}{(\gamma_G-\lambda_i)^l} \geq \frac{1}{N}\frac{1}{(\gamma_G-\lambda_1)^l} \geq N^{l-1-4l\epsilon}.
\]

For the upper bound, we split the sum into the first $N^{3\epsilon}$ eigenvalues and the remaining eigenvalues. By the lower bound in Lemma~\ref{lem:critical-point-G},
\[
\gamma_G-\lambda_1\ge \frac{1}{3\beta NT}.
\]
Therefore
\[
\frac1N\sum_{i=1}^{N^{3\epsilon}} \frac1{(\gamma_G-\lambda_i)^l} \le \frac1N\sum_{i=1}^{N^{3\epsilon}} \frac1{(\gamma_G-\lambda_1)^l} \le C^l N^{l-1+3\epsilon}.
\]
For $i>N^{3\epsilon}$, rigidity and the edge estimate for the classical locations imply
\[
\lambda_1-\lambda_i \ge c\,i^{2/3}N^{-2/3-\epsilon}
\]
with overwhelming probability. Since $\gamma_G>\lambda_1$, we have
\[
\gamma_G-\lambda_i \ge \lambda_1-\lambda_i \ge c\,i^{2/3}N^{-2/3-\epsilon}.
\]
Thus
\[
\begin{aligned}
\frac1N\sum_{i>N^{3\epsilon}} \frac1{(\gamma_G-\lambda_i)^l} \le \frac{C^l}{N} \sum_{i>N^{3\epsilon}} \frac{N^{2l/3+l\epsilon}}{i^{2l/3}} \le C^l N^{l-1+3\epsilon},
\end{aligned}
\]
where we used $l\ge2$. Combining the two estimates, we obtain
\[
\frac1N\sum_{i=1}^N \frac1{(\gamma_G-\lambda_i)^l} \le C_0^l N^{l-1+3\epsilon}.
\]
\end{proof}

\subsubsection{Construction and properties of $\Gamma$}
\label{app:Gamma-properties}

{
{
We adapt the global graph construction of \cite[Lemma~30 and Appendix~A]{LeeLi2023} to the effective inverse temperature $\beta t$.
}
In this subsection, we fix $t>\beta_c/\beta$ and write
\[
G(z)=G(z;t)
=
2\beta t z-\frac1N\sum_{i=1}^N\log(z-\lambda_i),
\]
where the logarithm is taken in the principal branch. Let $\gamma_G>\lambda_1$ be the unique real critical point of $G$.

{
We construct a contour
\[
\Gamma=\Gamma^-\cup\{\gamma_G\}\cup\Gamma^+
\]
with the following properties:
\begin{enumerate}
    \item[\rm(i)] $\Gamma^+$ is a $C^1$ curve contained in $\mathbb C^+$;
    \item[\rm(ii)] $\Gamma$ intersects the real axis only at $\gamma_G$;
    \item[\rm(iii)] $\Gamma^-$ is the reflection of $\Gamma^+$ across the real axis;
    \item[\rm(iv)] the tangent line of $\Gamma$ at $\gamma_G$ is parallel to the imaginary axis;
    \item[\rm(v)] $\Gamma^+$ has the global parametrization
    \[
    \Gamma^+
    =
    \left\{
    h_t(y)+\rm{i}y:
    0<y<\frac{\pi}{2\beta t}
    \right\},
    \]
    and its asymptote in the negative real direction is
    $y=\pi/(2\beta t)$;
    \item[\rm(vi)] $h_t(y)<\gamma_G$ for every
    $0<y<\pi/(2\beta t)$.
\end{enumerate}
}

For fixed
\[
0<y<\frac{\pi}{2\beta t},
\]
set
\[
H_y(x)
:=
\operatorname{Im}G(x+\rm{i}y)
=
2\beta t y
-\frac1N\sum_{i=1}^N
\operatorname{Arg}(x-\lambda_i+\rm{i}y),
\]
where
$0<\operatorname{Arg}(x-\lambda_i+\rm{i}y)<\pi$.
Then
\[
\lim_{x\to-\infty}H_y(x)=2\beta t y-\pi<0,
\qquad
\lim_{x\to+\infty}H_y(x)=2\beta t y>0,
\]
and
\[
\frac{\partial H_y}{\partial x}(x)
=
\frac{y}{N}\sum_{i=1}^N
\frac1{(x-\lambda_i)^2+y^2}>0.
\]
Hence, for each $y\in(0,\pi/(2\beta t))$, there exists a unique real number $h_t(y)$ such that 
\begin{equation}\label{eq:Gamma}
2\beta t y
=
\frac1N\sum_{i=1}^N
\operatorname{Arg}(h_t(y)-\lambda_i+\rm{i}y).
\end{equation}
Since $\partial_xH_y(h_t(y))>0$, the implicit function theorem implies $h_t\in C^1((0,\pi/(2\beta t)))$.

To identify the endpoint at $y=0$, define
\[
Q(x,y):=
\begin{cases}
\dfrac{\operatorname{Im}G(x+\rm{i}y)}{y},&y\neq0,\\[2mm]
G'(x),&y=0.
\end{cases}
\]
Since $G$ is analytic near $\gamma_G$, $Q$ is $C^1$ near $(\gamma_G,0)$, with
\[
Q(\gamma_G,0)=0,
\qquad
\partial_xQ(\gamma_G,0)=G''(\gamma_G)>0.
\]
{
Moreover, the expansion
\[
Q(x,y)
=
G'(x)-\frac{y^2}{6}G'''(x)+O(y^4)
\]
shows that
\[
\partial_yQ(\gamma_G,0)=0.
\]
}
The implicit function theorem therefore gives a unique local solution $x=h_t(y)$ satisfying
\[
h_t(0)=\gamma_G,
\qquad
h_t'(0)
=
-\frac{\partial_yQ(\gamma_G,0)}
{\partial_xQ(\gamma_G,0)}
=0.
\]
By uniqueness, this local solution agrees with the global graph above. Hence
\[
\Gamma^+
:=
\left\{
h_t(y)+\rm{i}y:
0<y<\frac{\pi}{2\beta t}
\right\}
\]
is a $C^1$ curve in $\mathbb C^+$. This proves property {\rm (i)}. Since $h_t(y)\to\gamma_G$ as $y\downarrow0$ and $y>0$ on $\Gamma^+$, the upper branch meets the real axis only at $\gamma_G$. This proves property {\rm (ii)}. The identity $h_t'(0)=0$ proves property {\rm (iv)}.

{
We next prove property {\rm (vi)}. For $y>0$,
\[
\frac{\partial}{\partial y}
\operatorname{Im}G(\gamma_G+\rm{i}y;t)
=
2\beta t
-
\frac1N\sum_{i=1}^N
\frac{\gamma_G-\lambda_i}
{(\gamma_G-\lambda_i)^2+y^2}.
\]
Since
\[
2\beta t
=
\frac1N\sum_{i=1}^N\frac1{\gamma_G-\lambda_i},
\]
the derivative is strictly positive. Hence
\[
H_y(\gamma_G)
=
\operatorname{Im}G(\gamma_G+\rm{i}y;t)>0.
\]
Because $H_y$ is strictly increasing in $x$ and $H_y(h_t(y))=0$, it follows that
\[
h_t(y)<\gamma_G,
\qquad
0<y<\frac{\pi}{2\beta t},
\]
which proves property {\rm (vi)}.
}

We now verify the steepest-descent property. Write
\[
G'(h_t(y)+\rm{i}y)=a(y)+\rm{i}b(y).
\]
Then
\[
b(y)
=
\frac{y}{N}\sum_{i=1}^N
\frac1{(h_t(y)-\lambda_i)^2+y^2}>0.
\]
Differentiating $\operatorname{Im}G(h_t(y)+\rm{i}y)=0$ gives
\[
h_t'(y)=-\frac{a(y)}{b(y)},
\]
and therefore
\[
\frac{d}{dy}\operatorname{Re}G(h_t(y)+\rm{i}y)
=
-\frac{a(y)^2+b(y)^2}{b(y)}<0.
\]
Hence $\operatorname{Re}G$ is strictly decreasing along $\Gamma^+$ as $y$ increases.

The parametrization gives
\[
0<y<\frac{\pi}{2\beta t}
\qquad\text{on }\Gamma^+.
\]
Moreover,
\[
h_t(y)\longrightarrow-\infty
\qquad\text{as }y\uparrow\frac{\pi}{2\beta t}.
\]
Indeed, if $y_n\uparrow\pi/(2\beta t)$ and $h_t(y_n)$ had a finite limit $x_0$, then \eqref{eq:Gamma} would imply
\[
\pi
=
\frac1N\sum_{i=1}^N
\operatorname{Arg}
\left(
x_0-\lambda_i+\frac{\rm{i}\pi}{2\beta t}
\right)
<\pi,
\]
a contradiction. The alternative $h_t(y_n)\to+\infty$ is also impossible because the right-hand side of \eqref{eq:Gamma} would tend to zero. Hence the asymptote of $\Gamma^+$ is
\[
y=\frac{\pi}{2\beta t}.
\]
This proves property {\rm (v)}.

For later use in the contour deformation, we also record the monotonicity of $h_t$ near its left end. From the implicit equation,
\[
-h_t'(y)
=
\frac{
2\beta t-\frac1N\sum_{i=1}^N
\frac{h_t(y)-\lambda_i}{(h_t(y)-\lambda_i)^2+y^2}
}{
\frac{y}{N}\sum_{i=1}^N
\frac1{(h_t(y)-\lambda_i)^2+y^2}
}.
\]
Using
\[
\left|\frac{r}{r^2+y^2}\right|
\leq\frac1{2y},
\qquad
\frac1{r^2+y^2}\leq\frac1{y^2},
\]
we obtain, for $y>1/(4\beta t)$,
\[
h_t'(y)\leq\frac12-2\beta t y.
\]
Hence, for every fixed $c_0\in(1/4,\pi/2)$, $h_t$ is strictly decreasing on
\[
\left[\frac{c_0}{\beta t},\frac{\pi}{2\beta t}\right).
\]
This also shows that the tail of $\Gamma^+$ can be parametrized by its real coordinate when needed.

Finally, since
\[
G(\overline z)=\overline{G(z)},
\]
the lower branch $\Gamma^-$ is the reflection of $\Gamma^+$ across the real axis. This proves property {\rm (iii)}. We orient
\[
\Gamma=\Gamma^-\cup\{\gamma_G\}\cup\Gamma^+
\]
from
$-\infty-\frac{\pi\rm{i}}{2\beta t}$ through $\gamma_G$ to
$-\infty+\frac{\pi\rm{i}}{2\beta t}$.
}

\subsubsection{Local estimates near $\gamma_G$}
\label{app:local-contour-estimates}

{
Let $\Gamma^+$ be the upper branch constructed in Appendix~\ref{app:Gamma-properties}. Recall $G'(\gamma_G;t)=0$. For $w=z-\gamma_G$, 
\begin{equation}\label{eq:G-local-Taylor}
G(\gamma_G+w;t)-G(\gamma_G;t)
=
\sum_{\ell=2}^{\infty}
\frac{G^{(\ell)}(\gamma_G;t)}{\ell!}w^\ell .
\end{equation}
By Lemma~\ref{lem:G-approximation}, uniformly for $t\in[\beta_c/\beta+\delta,T]$,
\begin{equation}\label{eq:G-derivative-local-bounds}
G''(\gamma_G;t)\geq N^{1-8\epsilon},
\qquad
\frac{|G^{(\ell)}(\gamma_G;t)|}{\ell!}
\leq
\frac{C_0^\ell}{\ell}N^{\ell-1+3\epsilon},
\qquad \ell\geq2 .
\end{equation}

Write
\[
z-\gamma_G=u+\rm{i}v,
\qquad z\in\Gamma^+,\quad v>0, \quad B_r = \{z \in \C : |z - \gamma_G| < r\}.
\]
We claim that throughout $\Gamma^+\cap B_{N^{-2}}$,
\begin{equation}\label{eq:Gamma-local-cone}
|u|\leq\frac v2.
\end{equation}
By property {\rm (iv)}, this holds sufficiently close to $\gamma_G$. If it failed, let $u+\rm{i}v$ be the first point, in the direction of increasing $v$, where $|u|=v/2$. Since $\operatorname{Im}G=0$ on $\Gamma$ and $G''(\gamma_G;t)$ is real,
\[
G''(\gamma_G;t)uv
=
-\operatorname{Im}
\sum_{\ell=3}^{\infty}
\frac{G^{(\ell)}(\gamma_G;t)}{\ell!}
(u+\rm{i}v)^\ell .
\]
At this point $v\asymp|u+\rm{i}v|$, and \eqref{eq:G-derivative-local-bounds} gives
\[
|u|
\leq
C N^{1+11\epsilon}|u+\rm{i}v|^2.
\]
Since $|u+\rm{i}v|\leq N^{-2}$,
\[
\frac{|u|}{v}
\leq C N^{-1+11\epsilon}=o(1),
\]
contradicting $|u|=v/2$. This proves
\eqref{eq:Gamma-local-cone}.

Let $z_{\rm loc}\in\Gamma^+$ be the first point, in the direction of increasing imaginary coordinate, such that
\[
|z_{\rm loc}-\gamma_G|=N^{-3},
\]
and write
\[
z_{\rm loc}-\gamma_G=u_{\rm loc}+\rm{i}y_{\rm loc}.
\]
By \eqref{eq:Gamma-local-cone},
\[
|u_{\rm loc}|\leq\frac{y_{\rm loc}}2,
\]
and hence
\begin{equation}\label{eq:zloc-imaginary-lower}
y_{\rm loc}
\geq
\frac{2}{\sqrt5}N^{-3}
\geq cN^{-3}.
\end{equation}
Furthermore,
\begin{align}
0
\leq
G(\gamma_G;t)-G(z_{\rm loc};t)
&\leq
|G(z_{\rm loc};t)-G(\gamma_G;t)| \nonumber\\
&\leq
\sum_{\ell=2}^{\infty}
\frac{C_0^\ell}{\ell}
N^{\ell-1+3\epsilon}N^{-3\ell}
\leq
2C_0^2N^{-5+3\epsilon}.
\label{eq:zloc-G-bound}
\end{align}
All estimates are uniform for
$t\in[\beta_c/\beta+\delta,T]$.
}

\subsubsection*{Proof of Lemma~\ref{lem:estimate-exp-G}}
\begin{proof}
We first define
\[
K_N(t) := -\rm{i} e^{-\frac{N}{2}G(\gamma_G;t)} \int_{\gamma_G-\rm{i}\infty}^{\gamma_G+\rm{i}\infty} e^{\frac{N}{2}G(z;t)}\,dz .
\]
Then \eqref{eq:estimate-exp-G} follows immediately from this definition. Since the contour integral term in $K_N(t)$ represents the spherical integral we saw in Lemma \ref{lem:integral-representation}, it directly follows that $K_N(t)$ is a positive real number.  It remains to prove the polynomial bounds on $K_N(t)$.

Put $z=\gamma_G+\rm{i}u$. Since $dz=\rm{i}\,du$, we have
\begin{align}
    K_N(t)
    &=
    \int_{-\infty}^{\infty}
    \exp\left\{
        \frac{N}{2}
        \bigl(G(\gamma_G+\rm{i}u;t)-G(\gamma_G;t)\bigr)
    \right\}\,du \nonumber \\
    &=
    \int_{-\infty}^{\infty}
    \exp\left\{
        iN\beta t u
        -
        \frac{1}{2}\sum_{j=1}^N
        \log\left(
            1+\frac{iu}{\gamma_G-\lambda_j}
        \right)
    \right\}\,du . \label{eq:KN-vertical-expression}
\end{align}
We work on the overwhelming probability event on which Condition \ref{cond:rigidity-eigenvalues}, Lemma \ref{lem:critical-point-G}, and Lemma \ref{lem:G-approximation} hold.

We first prove the upper bound. Taking absolute values in \eqref{eq:KN-vertical-expression}, we obtain
\[
|K_N(t)| \leq \int_{-\infty}^{\infty} \exp\left\{ -\frac{1}{4} \sum_{j=1}^N \log\left( 1+\frac{u^2}{(\gamma_G-\lambda_j)^2} \right) \right\}\,du .
\]
{
To make the uniformity in $t$ explicit, the critical-point equation gives
\[
2\beta t
=
\frac1N\sum_{j=1}^N\frac1{\gamma_G-\lambda_j}
\leq
\frac1{\gamma_G-\lambda_1},
\]
and hence
\[
\gamma_G-\lambda_1
\leq
\frac1{2\beta t}
\leq
\frac1{2\beta_c}.
\]
Together with the boundedness of spectrum with overwhelming probability, this implies
\[
\gamma_G-\lambda_j\leq C
\qquad (1\leq j\leq N)
\]
uniformly for $t>\beta_c/\beta$. Therefore
\[
|K_N(t)|
\leq
\int_{-\infty}^{\infty}(1+cu^2)^{-N/4}\,du
\leq C.
\]
}
Thus
\[
|K_N(t)|<C
\]
with overwhelming probability, uniformly for $t>\beta_c/\beta$.

{
We now prove the lower bound. We deform the vertical contour to the contour $\Gamma$ constructed in Appendix~\ref{app:Gamma-properties}. On a circular arc $|z|=R$ with $\operatorname{Re}z\leq\gamma_G$,
\[
\operatorname{Re}G(z;t)
\leq
2\beta T|\gamma_G|-\log(R/2)
\]
for large $R$, uniformly for
$t\in[\beta_c/\beta+\delta,T]$, so the arc contribution tends to zero. The tail monotonicity of $h_t$ proved in Appendix~\ref{app:Gamma-properties} also ensures absolute convergence along the two ends of $\Gamma$. Indeed, for $y\geq c_0/(\beta t)$ one has $h_t'(y)\leq-c<0$, so the tail can be parametrized by $x=h_t(y)$ with $|dy/dx|\leq C$. Moreover, as $x\to-\infty$ along $\Gamma^+$,
\[
\operatorname{Re}G(x+\rm{i}y;t)
=
2\beta t x-\frac1N\sum_{i=1}^N
\log|x-\lambda_i+\rm{i}y|
=
2\beta t x-\log|x|+O(1),
\]
and hence the tail integral is absolutely convergent. The same holds for $\Gamma^-$ by symmetry. Therefore
\[
\int_{\gamma_G-\rm{i}\infty}^{\gamma_G+\rm{i}\infty}
e^{\frac N2G(z;t)}\,dz
=
\int_{\Gamma}e^{\frac N2G(z;t)}\,dz.
\]
Since $\operatorname{Im}G=0$ on $\Gamma$, symmetry gives
\begin{equation}\label{eq:KN-gamma-plus}
K_N(t)
=
2\int_{\Gamma^+}
\exp\left\{
\frac N2\bigl(G(z;t)-G(\gamma_G;t)\bigr)
\right\}\,dy.
\end{equation}
The variable $y$ is a global increasing parameter on $\Gamma^+$, so no cancellation occurs in this integral.

Let $z_{\rm loc}$ and $y_{\rm loc}=\operatorname{Im}z_{\rm loc}$ be as in Appendix~\ref{app:local-contour-estimates}. By \eqref{eq:zloc-imaginary-lower},
\[
y_{\rm loc}\geq cN^{-3},
\]
and by \eqref{eq:zloc-G-bound},
\[
G(\gamma_G;t)-G(z_{\rm loc};t)
\leq
2C_0^2N^{-5+3\epsilon}.
\]
Since $\operatorname{Re}G$ decreases as $y$ increases along $\Gamma^+$,
\[
G(h_t(y)+\rm{i}y;t)\geq G(z_{\rm loc};t),
\qquad 0\leq y\leq y_{\rm loc}.
\]
Therefore
\begin{align*}
\frac12K_N(t)
&\geq
\int_0^{y_{\rm loc}}
\exp\left\{
\frac N2
\bigl(G(h_t(y)+\rm{i}y;t)-G(\gamma_G;t)\bigr)
\right\}\,dy\\
&\geq
y_{\rm loc}
\exp\left\{
\frac N2
\bigl(G(z_{\rm loc};t)-G(\gamma_G;t)\bigr)
\right\}\\
&\geq
cN^{-3}
\exp\{-C_0^2N^{-4+3\epsilon}\}.
\end{align*}
Thus, after changing $C>0$ if necessary,
\[
K_N(t)>N^{-C}
\]
with overwhelming probability, uniformly for
$t\in[\beta_c/\beta+\delta,T]$.
}
\end{proof}

\bibliographystyle{references_style}
\bibliography{references}

\end{document}